\documentclass[12pt,a4paper]{article}%
\usepackage{amsmath}
\usepackage{amssymb}
\usepackage{makeidx}
\usepackage{amsfonts}
\usepackage{hyperref}
\usepackage{graphicx}%
\providecommand{\U}[1]{\protect\rule{.1in}{.1in}}
\hypersetup{
colorlinks=true,
linkcolor=blue,
anchorcolor=blue,
citecolor=blue
}
\newtheorem{theorem}{Theorem}[section]

\newtheorem{definition}[theorem]{Definition}
\newtheorem{example}[theorem]{Example}
\newtheorem{lemma}[theorem]{Lemma}

\newtheorem{remark}[theorem]{Remark}

\newenvironment{proof}[1][Proof]{\noindent\textbf{#1.} }{\rule{0.5em}{0.5em}}

\begin{document}

\title{On the spectrum of the hierarchical Schr\"{o}dinger type operators }
\author{Alexander Bendikov\thanks{A. Bendikov was supported by the Polish National
Science center, grant 2015/17/B/ST1/00062 and by SFB 1283 of the German
Research Council. }
\and Alexander Grigor'yan\thanks{A. Grigor'yan was supported by SFB 1283 of the
German Research Council. }
\and Stanislav Molchanov\thanks{S. Molchanov was supported by the Russian Science
Foundation (Projects: 20-11-20119 and 17-11-01098). }}
\maketitle

\begin{abstract}
The goal of this paper is the spectral analysis of the Schr\"{o}dinger type
operator $H=L+V$, the perturbation of the Taibleson-Vladimirov multiplier
$L=\mathfrak{D}^{\alpha}$ by a potential $V$. Assuming that $V$ belongs to a
certain class of potentials we show that the discrete part of the spectrum of
$H$ may contain negative energies, it also appears in the spectral gaps of
$L$. We will split the spectrum of $H$ in two parts: high energy part
containing eigenvalues which correspond to the eigenfunctions located on the
support of the potential $V,$ and low energy part which lies in the spectrum
of certain bounded Schr\"{o}dinger-type operator acting on the Dyson
hierarchical lattice.

We pay special attention to the class of sparse potentials. In this case we
obtain precise spectral asymptotics for $H$ provided the sequence of distances
between locations tends to infinity fast enough.

We also obtain certain results concerning localization theory for $H$ subject
to (non-ergodic) random potential $V$. Examples illustrate our approach.

\end{abstract}
\tableofcontents

\section{Introduction}

\setcounter{equation}{0}

The spectral theory of nested fractals similar to the Sierpinski gasket, i.e.
the spectral theory of the corresponding Laplacians, is well understood. It
has several important features: Cantor-like structure of the essential
spectrum and, as result, the large number of spectral gaps, presence of
infinite number of eigenvalues each of which has infinite multiplicity and
compactly supported eigenstates, non-regularly varying at infinity heat
kernels which contain an oscillating in $\log t$ scale terms etc, see
\cite{GrabnerWoess}, \cite{DerfelGrabner} and \cite{BCW}.

The spectral properties mentioned above occur in the very precise form for the
Taibleson-Vladimirov Laplacian $\mathfrak{D}^{\alpha}$, the operator of
fractional derivative of order $\alpha$. This operator can be introduced in
several different forms (say, as $L^{2}$-multiplier in the $p$-adic analysis
setting, see \cite{Vladimirov}) but we select the geometric approach
\cite{Dyson1}, \cite{Molchanov}, \cite{Molchanov1}, \cite{BGP}, \cite{BGPW},
\cite{BendikovKrupski} and \cite{BGMS}.

\subsection{The Dyson hierarchical model}

Let us fix an integer $p\geq2$ and consider the family $\{\Pi_{r}%
:r\in\mathbb{Z}\}$ of partitions of the set $X=\mathbb{[}0,+\infty\lbrack$
such that each $\Pi_{r}$ consists of all intervals $I=[kp^{r},(k+1)p^{r}[$,
$k=0,1,...$. We call $r$ the rank of the partition $\Pi_{r}$\ (respectively,
the rank of the interval $I\in\Pi_{r}$). Each interval of rank $r$ is the
union of $p$ disjoint intervals of rank $(r-1)$. Each point $x\in X$ \ belongs
to a certain interval $I_{r}(x)$ of rank $r$, and intersection of all
intervals $I_{r}(x)$ is $\{x\}.$

\begin{definition}
Let $\mathcal{B}$ be the family of all intervals $[kp^{r},(k+1)p^{r}[$.
\emph{The hierarchical distance} $\mathrm{d}(x,y)$ is defined as the Lebesgue
measure $m(I)$ of the minimal interval $I\in\mathcal{B}$ which contains both
$x$ and $y$.
\end{definition}

It is easy to see that the function $(x,y)\rightarrow\mathrm{d}(x,y)$ is
non-degenerate, symmetric and for arbitrary $x,y$ and $z$,
\[
\mathrm{d}(x,y)\leq\max\{\mathrm{d}(x,z),\mathrm{d}(z,y)\},
\]
\ i.e. $\mathrm{d}(x,y)$ is an \emph{ultrametric} on $X$. It has the following properties:

\begin{itemize}
\item The ultrametric $\mathrm{d}(x,y)$ strictly majorizes the Euclidean
metric $|x-y|$. Indeed, by the very definition, $\mathrm{d}(x,y)\geq|x-y|$ for
all $x,y\in X$ whereas $\mathrm{d}(1-\varepsilon,1)=p$ for all $0<\varepsilon
<1$.
\end{itemize}

%

\begin{figure}
[tbh]
\begin{center}
\includegraphics[
height=2.7in,
width=4.8in
]%
{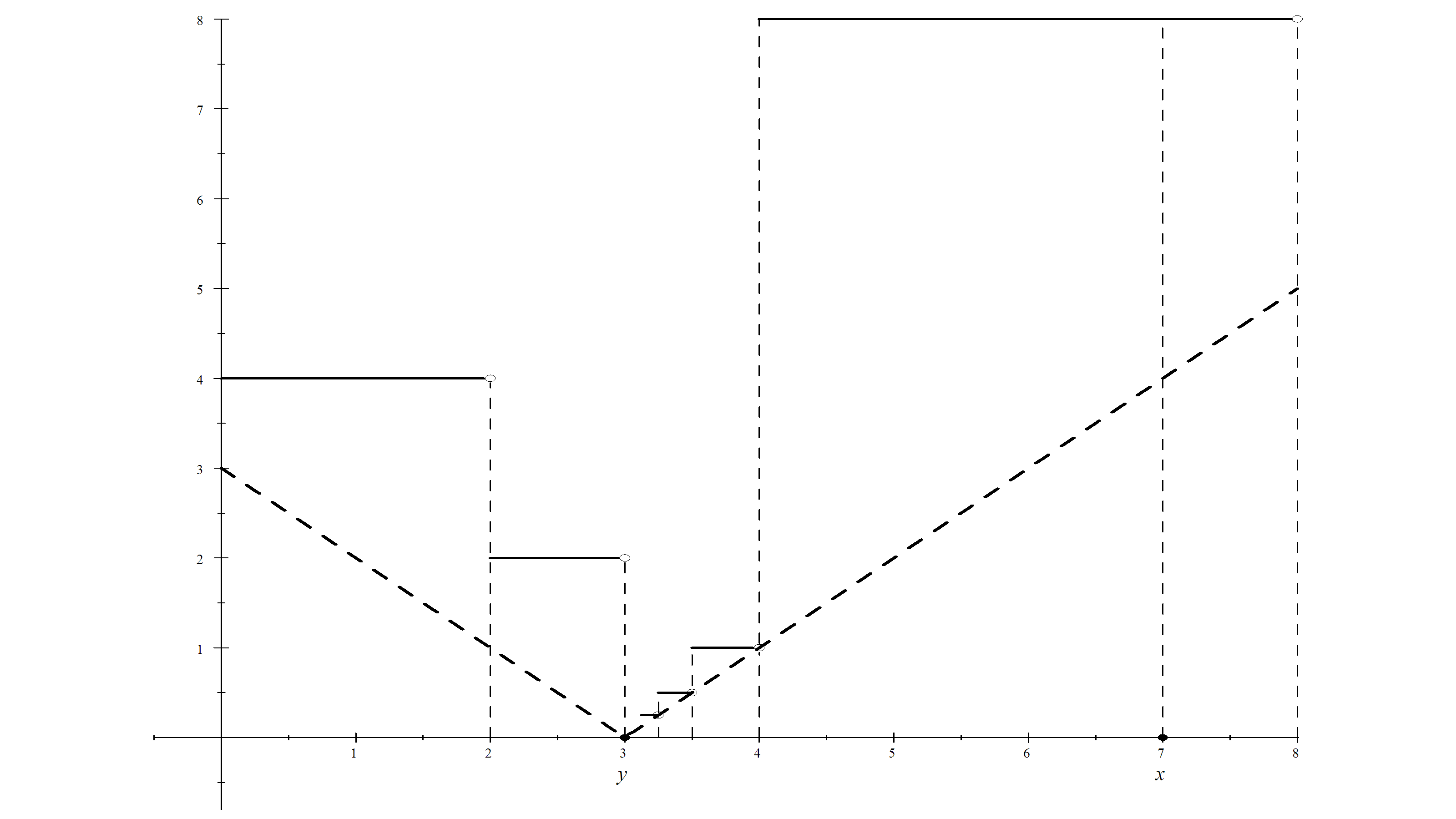}%
\caption{Comparison of two metrics: $d\left(  x,y\right)  \geq\left\vert
x-y\right\vert $}%
\label{pic1}%
\end{center}
\end{figure}

\begin{itemize}
\item The couple $(X,\mathrm{d})$ is a complete, locally compact, non-compact,
perfect and separable metric space. In this metric space the set of all open
balls coincides with the set of all intervals $I\in\mathcal{B}$. In
particular, in the metric space $(X,\mathrm{d})$ the set of all open balls is
countable whereas the set $X$ by itself is uncountable. Next property says
that $(X,\mathrm{d})$ is a totally disconnected metric space. \footnote{In
particular, $(X,\mathrm{d})$ is homeomorphic to the punctured Cantor set
$\{0,1\}^{\aleph_{0}}\backslash\{o\},$ see a survey on totally disconnected
metric spaces in \cite[Proposition 2.2]{BendikovKrupski}.}

\item Each open ball $B$ in $(X,\mathrm{d})$ is a closed set, each point $x$
of $B$ can be regarded as its center, any two balls $C$ and $D$ either do not
intersect or one is a subset of another etc.

\item It is remarkable that the Borel $\sigma$-algebra generated by the
ultrametric $\mathrm{d}(x,y)$ coincides with the classical Borel $\sigma
$-algebra (generated by the Euclidean metric).
\end{itemize}

\begin{definition}
Let us fix a parameter $\kappa\in]0,1[$. \emph{The hierarchical Laplacian} $L$
we introduce following \cite{MolchanovVainberg1} as a linear combination of
"elementary Laplacians"%
\begin{equation}
(Lf)(x)=%
{\displaystyle\sum\limits_{r=-\infty}^{+\infty}}
(1-\kappa)\kappa^{r-1}\left(  f(x)-\frac{1}{m(I_{r}(x))}%
{\displaystyle\int\limits_{I_{r}(x)}}
fdm\right)  .\text{ \ } \label{HL}%
\end{equation}

\end{definition}

The series in (\ref{HL}) diverges in general but it is finite and belongs to
$L^{2}(X,m)$ for any $f\in L^{2}(X,m)$\ which takes constant values on
intervals of any fixed rank $r$.

The operator $L$ admits a complete system of compactly supported
eigenfunctions. Indeed, let $I$ be an interval of rank $r$, and $I_{1}%
,I_{2},...,I_{p}$ be its subintervals of rank $r-1$. Let us consider $p$
functions%
\[
f_{I_{i}}=\frac{1_{I_{i}}}{m(I_{i})}-\frac{1_{I}}{m(I)},\text{ }i=1,2,...,p.
\]
Each function $f_{I_{i}}$\ belongs to the domain of the operator $L$\ and%
\[
Lf_{I_{i}}=\lambda(I)f_{I_{i}}\text{, where }\lambda(I)=\kappa^{r-1}.\text{\ }%
\]
Let us consider the eigenspace $\mathcal{H}(I):=\mathrm{span}\{f_{I_{i}}\}$.
Then $\dim\mathcal{H}(I)=p-1,$ the eigenspaces $\mathcal{H}(I)$ and
$\mathcal{H}(I^{\prime})$ are orthogonal for $I\neq I^{\prime}$ and%
\[%
{\displaystyle\bigoplus\limits_{r\in\mathbb{Z}}}
{\displaystyle\bigoplus\limits_{I\in\Pi_{r}}}
\mathcal{H}(I)=L^{2}(X,m).
\]
In particular, $L$ is essentially self-adjoint operator having a pure point
spectrum
\[
Spec(L)=\{0\}\cup\{\kappa^{r}:r\in\mathbb{Z}\}.
\]
Clearly each eigenvalue $\lambda(I)=\kappa^{r-1}$ has infinite multiplicity,
whence $Spec(L)$ coincides with its essential part $Spec_{ess}(L)$.

We shell see below that writing $\kappa=p^{-\alpha}$ the operator $L$ can be
identified with the Taibleson-Vladimirov operator $\mathfrak{D}^{\alpha}$, the
operator of fractional derivative of order $\alpha$, defined as $L^{2}%
$-multiplier in the $p$-adic analysis setting \cite{Vladimirov94},
\cite{Kochubey2004}.

The constant $D=2/\alpha$ is called \emph{the spectral dimension }(equations
(\ref{Dyson_heat_kernel}) and (\ref{Dyson_diagonal_h.k.}) show the role of
this constant in the heat kernel estimates below).

According to \cite{BGPW} the operator $L$ can be represented as a
hypersingular integral operator%
\[
Lf(x)=%
{\displaystyle\int\limits_{0}^{\infty}}
\left(  f(x)-f(y)\right)  J(x,y)dy
\]
where%
\[
J(x,y)=\frac{\kappa^{-1}-1}{1-\kappa/p}\cdot\frac{1}{\mathrm{d}(x,y)^{1+2/D}%
}.
\]
The Markov semigroup $(e^{-tL})_{t>0}$ is symmetric and admits a continuous
heat kernel $p(t,x,y)$ \footnote{$p(t,x,y)$ is continuous w.r.t. the
ultrametric \textrm{d}$(x,y)$ but it is discontinuous w.r.t. the Euclidean
metric $\left\vert x-y\right\vert $}. The function $p(t,x,y)$ can be estimated
(uniformly in $t,x$ and $y$) as follows
\begin{equation}
p(t,x,y)\asymp\frac{t}{[t^{D/2}+\mathrm{d}(x,y)]^{1+2/D}}.\text{
\ }\footnote{{} We write $f\asymp g$ if the ratio $f/g$ is bounded from above
and from below by positive constants for a specified range of variables. We
write $f\sim g$ if the ratio $f/g$ tends to identity.}\text{\ }
\label{Dyson_heat_kernel}%
\end{equation}
The function $p(t,x,x)$ does not depend on $x,$ denote it $p(t)$. It can be
represented in the form
\begin{equation}
p(t)=t^{-D/2}\mathcal{A(}\log_{2}t\mathcal{)}, \label{Dyson_diagonal_h.k.}%
\end{equation}
where $\mathcal{A(\tau)}$ is a continuous non-constant $\alpha$-periodic
function, see \cite[Proposition 2.3]{MolchanovVainberg1}, \cite{BGPW},
\cite{BCW}. In particular, in contrary to the classical case (symmetric stable
densities), the function $t\rightarrow p(t)$ does not vary regularly.

There are already several publications on the hierarchical Laplacian acting on
a general ultrametric measure space $(X,d,m)$ \cite{AlbeverioKarwowski},
\cite{AisenmanMolchanov}, \cite{Molchanov}, \cite{Molchanov1}, \cite{BGP},
\cite{BGPW}, \cite{BendikovKrupski}, \cite{BGMS}. By the general theory
developed in \cite{BGP}, \cite{BGPW} and \cite{BendikovKrupski}, any
hierarchical Laplacian $L$ acts in $L^{2}(X,m)$ as essentially self-adjoint
operator having a pure point spectrum. This operator can be represented in the
form
\begin{equation}
Lf(x)=%
{\displaystyle\int\limits_{X}}
(f(x)-f(y))J(x,y)dm(y)\text{. \ } \label{Spectrum}%
\end{equation}
The Markov semigroup $(e^{-tL})_{t>0}$ admits with respect to $m$ a continuous
transition density $p(t,x,y)$. It turns out that in terms of certain
(intrinsically related to $L$) ultrametric $d_{\ast}$,%
\begin{equation}
\text{\ }J(x,y)=\int\limits_{0}^{1/\emph{d}_{\ast}(x,y)}N(x,\tau)d\tau,
\label{Jump-kernel}%
\end{equation}%
\begin{equation}
\emph{p}(t,x,y)=t\int\limits_{0}^{1/\emph{d}_{\ast}(x,y)}N(x,\tau)\exp
(-t\tau)d\tau, \label{d*-jump kernel}%
\end{equation}
and%
\begin{equation}
\emph{p}(t,x,x)=\int\limits_{0}^{\infty}\exp(-t\tau)dN(x,\tau)
\label{on-diagonal density}%
\end{equation}
where $N(x,\tau)$ is the so called \emph{spectral function} related to $L$
(will be defined later).

\subsection{Outline}

Let us describe the main body of the paper. In Section 2 we introduce the
notion of homogeneous hierarchical Laplacian $L$ and list its basic
properties: the spectrum of the operator $L$ is pure point, all
eigenvalues of $L$ have infinite multiplicity and compactly supported
eigenfunctions, the heat kernel $p(t,x,y)$ exists and it is a continuous
function having certain asymptotic properties etc. As a special example we
consider the case $X=\mathbb{Q}_{p},$ the ring of $p$-adic numbers endowed
with its standard ultrametric $d(x,y)=\left\vert x-y\right\vert _{p}$ and the
normed Haar measure $m$. The hierarchical Laplacian $L$ in our example
coincides with the Taibleson-Vladimirov operator $\mathfrak{D}^{\alpha}$, the
operator of fractional derivative of order $\alpha$, see \cite{Vladimirov},
\cite{Vladimirov94}, and \cite{Kochubey2004}. The most complete source for the
basic definitions and facts related to the $p$-adic analysis is \cite{Koblitz}
and \cite{Taibleson75}.

The Schr\"{o}dinger type operator $H=L+V$ with hierarchical Laplacian $L$ was
studied in \cite{Dyson2}, \cite{Molchanov}, \cite{MolchanovVainberg1},
\cite{MolchanovVainberg2}, \cite{Bovier}, \cite{Kvitchevski1},
\cite{Kvitchevski2} (the hierarchical lattice of Dyson) and in
\cite{Vladimirov94}, \cite{VladimirovVolovich}, \cite{Kochubey2004} (the field
of $p$-adic numbers). In the next sections we consider the Schr\"{o}dinger
type operator acting on a homogeneous ultrametric space $X$. We assume that
the potential $V$ is of the form $V=%
{\displaystyle\sum}
\sigma_{i}1_{B_{i}}$, where $B_{i}$ are balls which belong to a fixed
horocycle $\mathcal{H}$ (i.e. all $B_{i}$ have the same diameter). The main
aim here is to study the set $Spec(H).$ Under certain assumptions on $V$ (e.g.
$V(x)\rightarrow0$ at infinity $\varpi$ etc.) we conclude that the set
$Spec(H)$\ is pure point (with possibly infinite number of limit points). We
split the set $Spec(H)$ in two disjoint parts: the first part consists of the
point $\lambda=0$ and the eigenvalues of the operator $L$ which correspond to
the horocycle $\mathcal{H}$ (with compactly supported eigenfunctions) and the
second part is the closure of a countably infinite set $\Xi$ of eigenvalues of
the operator $H$ (with non-compactly supported eigenfunctions). In the case of
sparse potential $V$, i.e. when $d(B_{i},B_{j})\rightarrow\infty$ fast enough
we specify the structure of the set $\Xi$. In this connection we would like to
mention here pioneering works of S. Molchanov \cite{Molchanov}, D. Krutikov
\cite{Krutikov1}, \cite{Krutikov2}, and N. Kochubei \cite{Kochubey2004}.

In the last section we consider the potential $V$ of the form $V=%
{\displaystyle\sum}
\sigma_{i}(\omega)1_{B_{i}}$, where $\sigma_{i}(\omega)$, $\omega\in
(\Omega,\digamma,P)$, are i.i.d. random variables, and embark on the
localization theory. More precisely, we show that if the sequence of
(non-random) distances $d(B_{i},B_{j})$ between locations tend to infinity
fast enough then the spectrum of $H$ is pure point for $P$-a.a. $\omega
\in\Omega$.

In the case when $X$ is discrete, $L$ is the Dyson Laplacian, $B_{i}$ are
singletons and $V$ is ergodic the \emph{localization theorem} appeared first
in the paper of Molchanov \cite{Molchanov} ($\sigma_{i}(\omega)$ are Cauchy
random variables) and later (under more general assumptions on $\sigma
_{i}(\omega)$) in the papers of Kritchevski \cite{Kvitchevski2} and
\cite{Kvitchevski1}. The proof of this theorem is based on the self-similarity
of $H$. This approach is not applicable to the case of (random) sparse potentials.

The proof of the localization theorem for (random) sparse potentials presented
in this paper is based on the abstract form of Simon-Wolff criterion
\cite{SimonWolff} for pure point spectrum, technique of fractional moments,
decoupling lemma of Molchanov and Borel-Cantelli type arguments, see
\cite{AisenmanMolchanov}, \cite{Molchanov1}.

\section{Preliminaries}

\setcounter{equation}{0}

\subsection{Homogeneous ultrametric space}

Let $(X,d)$ be an ultrametric space. Recall that a metric $d$ is called an
\emph{ultrametric} if it satisfies the ultrametric inequality
\[
d(x,y)\leq\max\{d(x,z),d(z,y)\},
\]
that is stronger than the usual triangle inequality. For any $x\in X$ and
$r\geq0$ consider the closed ball $B_{r}\left(  x\right)  =\left\{  y\in
X:d\left(  x,y\right)  \leq r\right\}  .$ The basic consequence of the
ultrametric property is that $B_{r}\left(  x\right)  $ is an open set for any
$r>0.$ Moreover, each point $y\in B_{r}\left(  x\right)  $ can be regarded as
its center, and any two balls of the same radius are either disjoint or
identical. This implies that, for any $0<s<r$, any ball $B_{r}\left(
x\right)  $ is a disjoint union of a finite family of balls of radius $s.$
Consequently, a collection of all distinct balls of the same radius form a
partition of $X.$ See e.g. \cite[Section 1]{BendikovKrupski} and references therein.

In this paper we always assume that the ultrametric space $(X,d)$ is
non-compact and that it is \emph{proper}, i.e. each $d$-ball is a compact set.
It follows that $\left(  X,d\right)  $ is separable. In addition to that we
always assume that $\left(  X,d\right)  $ is \emph{homogeneous} that is, the
group of isometries of $(X,d)$ acts transitively. In particular,\ a
homogeneous ultrametric measure space is either discrete or perfect.

Let $m$ be a Radon measure on $X$ with full support and such that $m$ is
invariant with respect to the the group of isometries of $(X,d)$. It follows
that any two balls of the same diameter have the same measure $m$ and that
$m\left(  X\right)  =\infty.$ Since the ultrametric property is preserved when
applying any monotone increasing function to $d$, we can and will assume
without loss of generality that, for all balls $B$ of positive diameter,
\[
m(B)=\mathrm{diam}(B).
\]
Let $\mathcal{B}$ be the family of all distinct balls of positive radii in
$X$. Since $m$ have full support, it follows that $m\left(  B\right)  >0$ for
all $B\in\mathcal{B}.$The set $\mathcal{B}$ is at most countable whereas $X$
by itself may well be uncountable (e.g. $X=[0,+\infty\lbrack$ with
$\mathcal{B}$ consisting of all $p$-adic intervals as explained in
Introduction). To any ultrametric space $(X,d)$ one can associate in a
standard fashion a tree $\mathcal{T}.$ The vertices of the tree are elements
of $\mathcal{B}$, the boundary $\partial\mathcal{T}$ can be identified with
the one-point compactification $X\cup\{\varpi\}$ of $X.$ We refer to
\cite{BendikovKrupski} for a treatment of the association between a
ultrametric space and the tree of its metric balls.

It is remarkable that a homogeneous ultrametric measure space $(X,d,m)$ can be
identified with  certain locally compact Abelian group $\mathfrak{G}$ equipped
with a translation invariant distance $\mathfrak{d}$ and the Haar measure
$\mathfrak{m}$, see the paper of Del Muto and Fig\`{a}-Talamanca \cite[Section
2]{Del Muto}. This identification is not unique. One possible way to define
such identification is to choose the sequence $\{a_{n}\}$\ of forward degrees
associated with the tree of balls $\Upsilon(X)$. This sequence is two-sided if
$X$ is non-compact and perfect, it is one-sided if $X$ is compact and perfect,
or if $X$ is discrete. In the 1st case we identify $X$ with $\Omega_{a}$, the
ring of $a$-adic numbers, in the 2nd case with $\Delta_{a}\subset\Omega_{a}$,
the ring of $a$-adic integers, and in the 3rd case with the discrete group
$\Omega_{a}/\Delta_{a}$. We refer to \cite[(10.1)-(10.11), (25.1),
(25.2)]{HewittRoss} for the comprehensive treatment of special groups
$\Omega_{a}$, $\Delta_{a}$ and $\Omega_{a}/\Delta_{a}$. The identification
$X\cong\mathfrak{G}$ makes it possible to use the Harmonic analysis tools
available for Abelian groups.

\subsection{Homogeneous hierarchical Laplacian}

\ Let $(X,d,m)$ be a homogeneous ultrametric space. Let $C:\mathcal{B}%
\rightarrow(0,\infty)$ be a function satisfying the following two conditions:

$(i)$ $C(A)=C(B)$ for any two balls $A$ and $B$ of the same diameter;

$(ii)$ for all $B\in\mathcal{B}$%
\begin{equation}
\lambda(B):=\sum\limits_{T\in\mathcal{B}:\text{ }B\subseteq T}C(T)<\infty.
\label{C1 condition}%
\end{equation}
The class of functions $C(B)$ satisfying $(i)$ and $(ii)$ is reach enough. For
example, fix $\alpha>0$ and set
\[
C(B)=m(B)^{-\alpha}-m(B^{\prime})^{-\alpha}%
\]
for any two nearing neighboring balls $B\subset B^{\prime}$
\textbf{\footnote{We say that $B\subset B^{\prime}$ are nearing neighboring
balls if for any $T\in\mathcal{B}$ such that $B\subseteq T\subseteq B^{\prime
}$ we have either $T=B$ or $T=B^{\prime}.$ }.} In this case it follows from
(\ref{C1 condition}) that
\[
\lambda(B)=m(B)^{-\alpha}.
\]
Let $\mathcal{D}$ be the set of all locally constant functions having compact
support. The set $\mathcal{D}$ belongs to Banach spaces $C_{0}(X)$ and
$L^{p}(X,m),$ $1\leq p<\infty,$ and is a dense subset there.

Denote by $\mathcal{B}\left(  x\right)  $ the family of all balls from
$\mathcal{B}$ containing $x$ (or equivalently, the family of all balls
$B_{r}\left(  x\right)  $ with $r>0$). Given the data $(\mathcal{B},C,m)$ we
define \emph{the} \emph{homogeneous} \emph{hierarchical} \emph{Laplacian} $L$
as an operator acting on functions $f\in\mathcal{D}$ as follows
\begin{equation}
Lf(x):=\sum\limits_{B\in\mathcal{B}(x)}C(B)\left(  f(x)-\frac{1}{m(B)}%
\int\limits_{B}fdm\right)  \text{.} \label{hlaplacian}%
\end{equation}
It is easy to see that $Lf\in L^{2}(X,m)$ so that we consider $(L,\mathcal{D}%
)$ as a densely defined operator in $L^{2}\left(  X,m\right)  .$ This operator
is symmetric and admits a complete system of eigenfunctions%
\begin{equation}
f_{B}=\frac{\mathbf{1}_{B}}{m(B)}-\frac{\mathbf{1}_{B^{\prime}}}{m(B^{\prime
})}, \label{eigenfunction}%
\end{equation}
where the couple $B\subset B^{\prime}$ runs over all nearest neighboring balls
from $\mathcal{B}$. The eigenvalue corresponding to $f_{B}$ is $\lambda
(B^{\prime})$ defined at (\ref{C1 condition}),
\[
Lf_{B}(x)=\lambda(B^{\prime})f_{B}(x).
\]
Since the system of eigenfunctions is complete, we conclude that
$(L,\mathcal{D})$ is essentially self-adjoint operator.

\emph{The intrinsic ultrametric} $d_{\ast}(x,y)$ is defined as follows
\begin{equation}
d_{\ast}(x,y):=\left\{
\begin{array}
[c]{ccc}%
0 & \text{when} & x=y\\
1/\lambda(x\curlywedge y) & \text{when} & x\neq y
\end{array}
\right.  , \label{intrinsic ultrametric}%
\end{equation}
where $x\curlywedge y$ is the minimal ball containing both $x$ and $y$. In
particular, for any open ball $B,$ we have%
\begin{equation}
\lambda(B)=\frac{1}{\mathrm{diam}_{\ast}(B)}. \label{intrinsic diameter}%
\end{equation}
\emph{The spectral function} $\tau\rightarrow N(\tau),$ see equation
(\ref{Jump-kernel}), is defined as a left-continuous step-function having
jumps at the points $\lambda(B)$, and%
\[
N(\lambda(B))=1/m(B).
\]
$\emph{The}$ \emph{volume function }$V(r)$ is defined by setting $V(r)=m(B)$
where the ball $B$ has $d_{\ast}$-radius $r$. It is easy to see that
\begin{equation}
N(\tau)=1/V(1/\tau). \label{Spectral function}%
\end{equation}
\emph{The Markov semigroup} $P_{t}=e^{-tL}$ admits a continuous density
$p(t,x,y)$ with respect to $m$, we call it \emph{the heat kernel.} The
function $p(t,x,y)$ can be represented in the form given by equations
(\ref{d*-jump kernel}) and (\ref{on-diagonal density}). Respectively, the
Markov generator $L$ admits the representation given by equations
(\ref{Spectrum}) and (\ref{Jump-kernel}).

\emph{The resolvent operator} $(L+\lambda\mathrm{I})^{-1},\lambda>0,$ admits a
continuous strictly positive kernel $\mathcal{R}(\lambda,x,y)$ with respect to
the measure $m$. The resolvent operator\ is well defined for $\lambda=0,$ i.e.
the Markov semigroup $(P_{t})_{t>0}$ is transient, if and only if for some
(equivalently, for all) $x\in X$ the function $\tau\rightarrow1/V(\tau)$ is
integrable at $\infty$. Its kernel $\mathcal{R}(0,x,y)$, called also \emph{the
Green function}, is of the form%
\begin{equation}
\mathcal{R}(0,x,y)=%
{\displaystyle\int\limits_{\emph{r}}^{+\infty}}
\frac{d\tau}{V(\tau)}\text{, }r=\emph{d}_{\ast}(x,y). \label{Green function}%
\end{equation}
Under certain Tauberian conditions the equation from above takes the form
\begin{equation}
\mathcal{R}(0,x,y)\asymp\frac{\emph{r}}{V(\emph{r})}\text{, }r=\emph{d}_{\ast
}(x,y). \label{Green function'}%
\end{equation}

\subsection{An example}

Let $\Phi:\mathbb{R}_{+}\rightarrow\mathbb{R}_{+}$ be an increasing
homeomorphism. For any two nearest neighboring balls $B\subset B^{\prime}$ we
define
\begin{equation}
C(B)=\Phi\left(  1/m(B)\right)  -\Phi\left(  1/m(B^{\prime})\right)  .
\label{An example}%
\end{equation}
Then the following properties hold:

\begin{description}
\item[(i)] $\lambda(B)=\Phi\left(  1/m(B)\right)  $,

\item[(ii)] \ $d_{\ast}(x,y)=1/\Phi\left(  1/m(x\curlywedge y)\right)  $,

\item[(iii)] $V(r)\leq1/\Phi^{-1}(1/r).$ Moreover, $V(r)\asymp1/\Phi
^{-1}(1/r)$ whenever both $\Phi$ and $\Phi^{-1}$ are doubling and
$m(B^{\prime})\leq cm(B)$ for some $c>0$ and all neighboring balls $B\subset
B^{\prime}$. In turn, this yields
\end{description}

\begin{equation}
p(t,x,y)\asymp t\cdot\min\left\{  \frac{1}{t}\Phi^{-1}\left(  \frac{1}%
{t}\right)  ,\frac{1}{m(x\curlywedge y)}\Phi\left(  \frac{1}{m(x\curlywedge
y)}\right)  \right\}  , \label{HK-Ex}%
\end{equation}
and%
\begin{equation}
p(t,x,x)\asymp\Phi^{-1}\left(  \frac{1}{t}\right)  \label{Spectral asympt.}%
\end{equation}
for all $t>0$ and $x,y\in X$.

\subsection{$\mathcal{L}^{2}$-multipliers}

As a special case of the general construction consider $X=\mathbb{Q}_{p}$, the
ring of $p$-adic numbers equipped with its standard ultrametric
$d(x,y)=\left\vert x-y\right\vert _{p}$. Notice that the ultrametric spaces
$(\mathbb{Q}_{p},d)$ and $(\mathbb{[}0,\infty\mathbb{)},\mathrm{d})$ with
non-Euclidean$\ \mathrm{d}$ (the Dyson's model) as explained in the
introduction, are isometric.

Let $m$ be the normed Haar measure on the Abelian group $\mathbb{Q}_{p},$
\ $\mathcal{L}^{2}=L^{2}(\mathbb{Q}_{p},m)$ and $\mathcal{F}:f\rightarrow
\widehat{f}$ the Fourier transform acting in $\mathcal{L}^{2}$. It is known,
see \cite{Taibleson75}, \cite{Vladimirov94}, \cite{Kochubey2004}, that
$\mathcal{F}:\mathcal{D}\rightarrow\mathcal{D}$ is a bijection.

Let $\Phi:\mathbb{R}_{+}\rightarrow\mathbb{R}_{+}$ be an increasing
homeomorphism. The self-adjoint operator $\Phi(\mathfrak{D)}$ we define as
$\mathcal{L}^{2}-$multiplier, that is,
\[
\widehat{\Phi(\mathfrak{D)}f}(\xi)=\Phi(\left\vert \xi\right\vert
_{p})\widehat{f}(\xi),\text{ \ }\xi\in\mathbb{Q}_{p}.
\]
By \cite[Theorem 3.1]{BGPW}, $\Phi(\mathfrak{D)}$ is a homogeneous
hierarchical Laplacian. The eigenvalues $\lambda(B)$\ of the operator
$\Phi(\mathfrak{D)}$ are of the form
\begin{equation}
\lambda(B)=\Phi\left(  \frac{p}{m(B)}\right)  . \label{Lambda-Phi eigenvalue}%
\end{equation}
\ Let $p(t,x,y)$ be the heat kernel associated with the operator
$\Phi(\mathfrak{D}).$ Assume that both $\Phi$ and $\Phi^{-1}$ are doubling,
then equations (\ref{HK-Ex}) and (\ref{Spectral asympt.}) apply. Since for any
$x,y\in\mathbb{Q}_{p}$, $m(x\curlywedge y)=\left\vert x-y\right\vert _{p}$ we
obtain
\begin{equation}
p(t,x,y)\asymp t\cdot\min\left\{  \frac{1}{t}\Phi^{-1}\left(  \frac{1}%
{t}\right)  ,\frac{1}{\left\vert x-y\right\vert _{p}}\Phi\left(  \frac
{1}{\left\vert x-y\right\vert _{p}}\right)  \right\}  , \label{HK-bounds II}%
\end{equation}
and%
\begin{equation}
p(t,x,x)\asymp\Phi^{-1}\left(  \frac{1}{t}\right)  . \label{J-bounds I}%
\end{equation}
The Taibleson-Vladimirov operator $\mathfrak{D}^{\alpha}$ is $\mathcal{L}^{2}%
$-multiplier, it can be written as a hypersingular integral operator
\begin{equation}
\mathfrak{D}^{\alpha}f(x)=\frac{1}{\Gamma_{p}(-\alpha)}\int_{\mathbb{Q}_{p}%
}\frac{f(y)-f(x)}{\left\vert y-x\right\vert _{p}^{1+\alpha}}dm(y),
\label{D_alpha}%
\end{equation}
where%
\[
\Gamma_{p}(z)=\frac{1-p^{z-1}}{1-p^{-z}}%
\]
is the $p$-adic Gamma-function \cite[Ch.1, Sec.VIII.2, Eq.(2.17)
]{Vladimirov94}.

The heat kernel $p_{\alpha}(t,x,y)$ of the operator $\mathfrak{D}^{\alpha}$
admits two-sided bounds
\begin{equation}
p_{\alpha}(t,x,y)\asymp\frac{t}{(t^{1/\alpha}+\left\vert x-y\right\vert
_{p})^{1+\alpha}}.\text{ } \label{alpha heat k}%
\end{equation}
In particular, $p_{\alpha}(t,x,x)\asymp t^{-1/\alpha}$, whence the Markov
semigroup $(e^{-t\mathfrak{D}^{\alpha}})_{t>0}$ is transient if and only if
$\alpha<1$.

In the transient case the Green function is of the form%
\begin{equation}
\mathcal{R}_{\alpha}(0,x,y)=\frac{1}{\Gamma_{p}(\alpha)}\frac{1}{\left\vert
x-y\right\vert _{p}^{1-\alpha}}.
\end{equation}
For all facts listed above we refer the reader to \cite{BGP}, \cite{BGPW} and
\cite{BendikovKrupski}.

\section{Schr\"{o}dinger type operators}

Let $(X,d,m)$ be a homogeneous ultrametric measure space and $L$ a homogeneous
hierarchical Laplacian acting on $(X,d,m)$. Identifying $(X,d)$ with a locally
compact Abelian group (say, $X=\mathbb{Q}_{a}$) we can regard $-L$ as a
translation invariant isotropic Markov generator. By (\ref{Spectrum}), the
operator $(L,\mathcal{D})$ is of the form%
\begin{equation}
Lf(x)=%
{\displaystyle\int\limits_{X}}
(f(x)-f(y))J(x-y)dm(y).\text{ } \label{Levy generator}%
\end{equation}
In terms of the Fourier transform, we have%
\[
\widehat{Lf}(\theta)=\widehat{L}(\theta)\cdot\widehat{f}(\theta),
\]
where
\begin{equation}
\widehat{L}(\theta)=%
{\displaystyle\int\limits_{X}}
[1-\mathfrak{\operatorname{Re}}\left\langle h,\theta\right\rangle ]J(h)dm(h).
\label{Levy symbol}%
\end{equation}
The function $\theta\rightarrow\widehat{L}(\theta)$ depends on the $a$-adic
distance $\left\Vert \theta\right\Vert _{a}$, and as a function of the
distance it is strictly increasing, zero at zero and infinity at infinity. In
particular, it satisfies the ultrametric inequality%
\[
\widehat{L}(\theta_{1}+\theta_{2})\leq\max\{\widehat{L}(\theta_{1}%
),\widehat{L}(\theta_{2})\}.
\]
Consider the Schr\"{o}dinger type operator
\begin{equation}
Hu=Lu+V\cdot u,\text{ \ } \label{Schroedinger operator}%
\end{equation}
where $V$ is a real measurable function (the potential). Our goal is to show
that under certain conditions on $V$ one may associate a self-adjoint operator
$H$ with the equation (\ref{Schroedinger operator}).

If the potential $V$ is locally bounded then $H$ $:\mathcal{D}\rightarrow
L^{2}(X,m)$ is a well-defined symmetric operator.

\begin{theorem}
\label{Schroedinger spectrum}Assume that $V$ is locally bounded. Then the
following is true:

$(i)$ The operator $H=L+V$ is essentially self-adjoint.\footnote{Recall that,
for the classical Schr\"{o}dinger operator $H=-\Delta+V$ in $\mathbb{R}^{n},$
this statement is \emph{not} true, unless $V$ satisfies a certain lower bound,
see \cite[Chapter II, Theorem 1.1 and Example 1.1]{BeresinShubin}.}

$(ii)$ Assume that $V(x)\rightarrow+\infty$ as $x\rightarrow\varpi$. Then the
operator $H$ has a compact resolvent. Consequently, the spectrum of $H$ is discrete.

$(iii)$ Assume that $V(x)\rightarrow0$ as $x\rightarrow\varpi$. Then the
essential spectrum of $H$ coincides with the spectrum of $L$. Thus, the
spectrum of $H$ is pure point and the negative part of the spectrum consists
of isolated eigenvalues of finite multiplicity.
\end{theorem}

\begin{proof}
$(i)$ Let us choose an open ball $O$ which contains the neutral element and
write equation (\ref{Levy generator}) in the form%
\begin{align*}
Lf(x)  &  =\left(
{\displaystyle\int\limits_{O}}
+%
{\displaystyle\int\limits_{O^{c}}}
\right)  [f(x)-f(x+y)]J(y)dm(y)\\
&  =L_{O}f(x)+L_{O^{c}}f(x).
\end{align*}
We have $Hf=L_{O}f+L_{O^{c}}f+Vf$, where the operator $V$ is the operator of
multiplication by the function $V(x)$. The operator $L_{O^{c}}f=J(O^{c}%
)(f-a\ast f)$, where $a(y)=J(y)1_{O^{c}}(y)/J(O^{c}),$ is a bounded symmetric
operator in $L^{2}(X,m)$\ (as $f\rightarrow a\ast f$ \ is the operator of
convolution with probability measure $a(y)dm(y)$) and thus does not influence
self-adjointness. As $L_{O}$ is minus L\'{e}vy generator it is essentially
self-adjoint (one more way to make this conclusion is that the matrix of the
operator $L_{O}$\ is diagonal in the basis $\{f_{B}\}$ of eigenfunctions of
the operator $L$, see \cite{Kozyrev}).

For any ball $B$ which belongs to the same horocycle $\mathcal{H}$ as $O$ we
denote $\mathfrak{H}_{B}$ the subspace of $L^{2}(X,m)$ which consists of all
functions $f$ \ having support in $B$. Since $O$ is a subgroup of the Abelian
group $X$ and each ball $B\in\mathcal{H}$ is a coset (i.e. belongs to the
quotient group $[X:O]$), we conclude that $\mathfrak{H}_{B}$ is an invariant
subspace of the symmetric operator $H_{O}=L_{O}+V$. Moreover, $\mathfrak{H}%
_{B}$ reduces $H_{O}$.

The ultrametric space $X$ can be covered by a sequence of non-intersecting
balls $B_{n}$ (recall that due to the ultrametric property two balls of the
same diameter either coincide or do not intersect). This leads to the
orthogonal decomposition%
\[
L^{2}(X,m)=%
{\displaystyle\bigoplus\limits_{n}}
\mathfrak{H}_{B_{n}}%
\]
where each $\mathfrak{H}_{B_{n}}$ reduces $H_{O}$. The restriction of the
essentially self-adjoint operator $L_{O}$ to its invariant subspace
$\mathfrak{H}_{B_{n}}$ is an essentially self-adjoint operator, while the
restriction of the operator $V$ is bounded. Thus $H_{O}$ is essentially
self-adjoint as orthogonal sum of essentially self-adjoint operators $H_{O,n}%
$, the restriction of $H_{O}$ to $\mathfrak{H}_{B_{n}}$.

$(ii)$ The proof is similar to the one for the Schr\"{o}dinger operators given
in \cite[Theorem X.3]{Vladimirov94}; the main tools are boundedness from below
of the operator $H$ and the Riesz-Rellich compactness criteria for subsets of
$L^{2}(X,m)$.

$(iii)$ Let us show that the operator $V$ is $L-$\emph{compact. }Then, by
\cite[Theorem IV.5.35]{Kato}, the essential spectrums of the operators $H$ and
$L$ coincide. Recall that $L-$compactness means that if a sequence $\{u_{n}\}$
is such that both $\{u_{n}\}$ and $\{Lu_{n}\}$ are bounded then there exists a
subsequence $\{u_{n}^{\prime}\}\subset\{u_{n}\}$ such that the sequence
$\{Vu_{n}^{\prime}\}$ converges.

1. Denote $v_{n}=Lu_{n}+u_{n}.$ By assumption the sequence $\{v_{n}\}$ is
bounded and $u_{n}=\mathcal{R}_{1}v_{n}=r_{1}\ast v_{n}$. It follows that the
quantity%
\[
\left(
{\displaystyle\int}
\left\vert u_{n}(x+h)-u_{n}(x)\right\vert ^{2}dm(x)\right)  ^{1/2}%
\leq\left\Vert v_{n}\right\Vert _{L^{2}}%
{\displaystyle\int}
\left\vert r_{1}(z+h)-r_{1}(z)\right\vert dm(z)
\]
tends to zero uniformly in $n$ as $h$ tends to the neutral element. Thus, the
sequence $\{u_{n}\}$ consists of equicontinuous on the whole in $L^{2}(X,m)$
functions. The same is true for the sequence $\{Vu_{n}\}$. Indeed, for any
ball $B$ which contains the neutral element we write%
\[
\left(
{\displaystyle\int}
\left\vert V(x+h)u_{n}(x+h)-V(x)u_{n}(x)\right\vert ^{2}dm(x)\right)
^{1/2}\leq I+II+III,
\]
where%
\[
I=\left\Vert V\right\Vert _{L^{\infty}}\left(
{\displaystyle\int}
\left\vert u_{n}(x+h)-u_{n}(x)\right\vert ^{2}dm(x)\right)  ^{1/2},
\]%
\[
II=\left\Vert u_{n}\right\Vert _{L^{2}}\left(
{\displaystyle\int_{B}}
\left\vert V(x+h)-V(x)\right\vert ^{2}dm(x)\right)  ^{1/2},
\]%
\[
III=\left\Vert u_{n}\right\Vert _{L^{2}}\sup_{x\in B^{c}}\left\vert
V(x+h)-V(x)\right\vert .
\]
Clearly $I,II$ and $III$ tend to zero uniformly in $n$ as $h$ tends to the
neutral element and $B\nearrow X$.

2. The sequence $\{Vu_{n}\}$ consists of functions with equicontinuous
$L^{2}(X,m)$ integrals at infinity. Indeed, for any ball $B$ which contains
the neutral element we have%
\[%
{\displaystyle\int\limits_{B^{c}}}
\left\vert Vu_{n}(x)\right\vert ^{2}dm(x)\leq\left\Vert u_{n}\right\Vert
_{L^{2}}\sup_{x\in B^{c}}\left\vert V(x)\right\vert \rightarrow0
\]
uniformly in $n$ as $B\nearrow X.$

Thus, the sequence $\{Vu_{n}\}$ is bounded in $L^{2}(X,m)$, consists of
equicontinuous on whole in $L^{2}(X,m)$ \ functions with equicontinuous
$L^{2}(X,m)$ integrals at infinity. By the Riesz-Kolmogorov criterion of
compactness in $L^{2}(X,m)$, the set $\{Vu_{n}\}$ is compact, whence it
contains a convergent subsequence $\{Vu_{n}^{\prime}\},$\ as claimed.
\end{proof}

In the case when the ultrametric measure space $(X,d,m)$ is \emph{countably
infinite }the statement (ii) of\emph{ }Theorem \ref{Schroedinger spectrum} can
be complemented as follows.

\begin{theorem}
\label{Schroedinger spectrum'}Assume that $(X,d,m)$ is \emph{countably
infinite. }Then the following statements are equivalent:

\begin{description}
\item[(i)] The operator $H$ has a discrete spectrum.

\item[(ii)] $|V(x)|$ tend to infinity as $x\rightarrow\varpi.$
\end{description}
\end{theorem}

\begin{proof}
$(ii)\Longrightarrow(i):$ Since $X$ is discrete $L$ is a bounded symmetric
operator, let us set $d:=\left\Vert L\right\Vert $. Suppose that $|V(x)|$ tend
to infinity as $x\rightarrow\varpi.$ Then for every given interval $I=[a,b]$
and its neighborhood $I^{\prime}=[a-d-1,b+d+1]$ there exist a finite set $A$
of points $x$ such that $V(x)\in I^{\prime}$. Let us choose $v\notin
I^{\prime}$ and define the operator $H^{\prime}=L+V^{\prime}$ where%
\[
V^{\prime}(x):=\left\{
\begin{array}
[c]{ccc}%
V(x) & \text{if} & x\notin A\\
v & \text{if} & x\in A
\end{array}
\right.  .
\]
The resolvent of the operator $V:u(x)\rightarrow V(x)u(x)$ is analytic inside
of $I^{\prime}$ and, as a result, the resolvent of $H^{\prime}$ is analytic
inside of $I$. Indeed, it is straightforward to show that
\[
\left\Vert L(V^{\prime}-\lambda\mathrm{I})^{-1}\right\Vert =\left\Vert
(V^{\prime}-\lambda\mathrm{I})^{-1}L\right\Vert \leq\frac{d}{d+1}<1,
\]
for any $\lambda\in I$. It follows that the operator%
\[
H^{\prime}-\lambda\mathrm{I}=(V^{\prime}-\lambda\mathrm{I})\left(
E+L(V^{\prime}-\lambda\mathrm{I})^{-1}\right)
\]
is invertible. This in turn implies that the operator $H^{\prime}$ has no
spectrum inside the interval $I$. But the difference $H-H^{\prime}$ is an
operator of finite rank. Hence the operator $H$ has (in the same interval $I$)
not more than finite number of eigenvalues, see Lemma \ref{Weyl perturbation}
below. Thus we have already proved that the spectrum of $H$ is discrete.

$(i)\Longrightarrow(ii):$ Suppose that the operator $H$ has a discrete
spectrum. Then clearly the spectrum of $H^{2}$ is also discrete. Let
$E_{1}\leq E_{2}\leq\cdots$ be the eigenvalues of $H^{2}$. Then by Courant's
$\min-\max$ principle%
\begin{equation}
E_{n}=\min_{\psi_{1},\ldots,\psi_{n}}\max\{(\psi,H^{2}\psi):\psi\in
span(\psi_{1},\ldots,\psi_{n}),\left\Vert \psi\right\Vert =1\}.
\label{Courant principle}%
\end{equation}
Assume that $|V(x)|$ does not tend to $+\infty$ as $x\rightarrow\varpi$. Then
there exists a sequence $\{x_{n}\}\subset X$ such that $\left\vert
V(x_{n})\right\vert \leq C$ for some $C>0$ and all $n\geq1$. It follows that%
\begin{equation}
(\psi,H^{2}\psi)\leq2(d^{2}+C^{2}),\text{ }\forall\psi\in span(\delta_{x_{1}%
},\delta_{x_{2}},\delta_{x_{3}},...),\left\Vert \psi\right\Vert =1.
\label{bounds}%
\end{equation}
Equations (\ref{Courant principle}) and (\ref{bounds}) imply that the interval
$[0,2(d^{2}+C^{2})]$ contains at list one limit point of the sequence
$\{E_{n}\}$, i.e. the essential spectrum of $H^{2}$ (equivalently of $H$) is
not empty. This fact contradicts the discreetness of the spectrum of $H^{2}$
(or $H$). This proves the second part of the theorem.
\end{proof}

In the continuous case the situation is not so obvious. In what follows we
restrict ourself by considering a class $\mathcal{K}$ of potentials of the
form $V=\sum_{B\in\mathcal{H}}\sigma(B)1_{B},$ where $\mathcal{H}$ is a fixed
horocycle. Let us select the following Hilbert subspaces of $L^{2}(X,m):$

$1.$ $\mathcal{L}_{+}=\mathrm{span}\{1_{B}:B\in\mathcal{H}\}$,

$2.$ $\mathcal{L}_{B}=\mathrm{span}\{f_{T}:T\varsubsetneq B\}$,

$3.$ $\mathcal{L}_{-}=L^{2}(X,m)\ominus\mathcal{L}_{+}=%
{\displaystyle\bigoplus_{B\in\mathcal{H}}}
\mathcal{L}_{B}$.

The following three lemmas can be proved by inspection.

\begin{lemma}
\label{Splitting lemma}The linear spaces $\mathcal{L}_{+},$ $\mathcal{L}_{B}$
and $\mathcal{L}_{-}$ are invariant subspaces for both operators $H$ and $L$.
Let $H_{+},H_{B}$ and $H_{-}$(resp. $L_{+},L_{B}$ and $L_{-}$) be the
restriction of the operator $H$ (resp. $L$) to $\mathcal{L}_{+},$
$\mathcal{L}_{B}$ and $\mathcal{L}_{-}$ respectively. The following properties
hold true:

$(i)$ $H=H_{+}\oplus H_{-}$,

$(ii)$ $H_{B}=L_{B}+\sigma(B)$,

$(iii)$ $H_{-}=%
{\displaystyle\bigoplus_{B\in\mathcal{H}}}
(L_{B}+\sigma(B))$.
\end{lemma}

Remind that $Lf_{B}=\lambda(B^{\prime})f_{B}$ for any open ball $B$. As $B$
converges to a singleton $\lambda(B^{\prime})\rightarrow+\infty$ whence
$L_{B}$ has discrete spectrum. By the homogenuity property $Spec(L_{A})$ is
the same for all $A$ $\in\mathcal{H}$, we denote it $\mathfrak{S}%
_{\mathcal{H}}$. We also set $\mathfrak{R}_{V}:\mathfrak{=}Range(V)$.

\begin{lemma}
\label{Splitting spectrum}In the notation from above%
\[
Spec(H_{-})=\overline{\mathfrak{S}_{\mathcal{H}}+\mathfrak{R}_{V}}.
\]
In particular, the operator $H_{-}$ has a pure point (not necessary discrete) spectrum.
\end{lemma}

Let us choose in each ball $B\in\mathcal{H}$ an element $a_{B}$ and consider a
discrete ultrametric space $(X^{\prime},m^{\prime},d^{\prime})$ with
$X^{\prime}=\{a_{B}:B\in\mathcal{H}\}$ induced by $(X,m,d)$.

\begin{lemma}
\label{H-plus lemma}The operator $L_{+}$ can be identified with certain
hierarchical Laplacian $L^{\prime}$ acting on $(X^{\prime},m^{\prime
},d^{\prime})$, respectively the operator $H_{+}$ can be identified with
certain Schr\"{o}dinger type operator $H^{\prime}=L^{\prime}+V^{\prime}$ with
potential $V^{\prime}=\sum_{a\in X^{\prime}}V(a)\delta_{a}$.
\end{lemma}

\begin{theorem}
For $V\in\mathcal{K}$ the statements $(i)$ and $(ii)$ of Theorem
\ref{Schroedinger spectrum'} are related by the implication
$(i)\Longrightarrow(ii)$. The inverse implication $(ii)\Longrightarrow(i)$
holds true if and only if the set $\mathfrak{S}_{\mathcal{H}}+\mathfrak{R}%
_{V}$ has no accumulating points.
\end{theorem}

\begin{proof}
If we assume that $Spec(H)$ is discrete, then the operator $H_{+}$ (whence the
operator $H^{\prime}$) has a discrete spectrum. Applying Theorem
\ref{Schroedinger spectrum'} we conclude that $|V(x)|\rightarrow+\infty$, i.e.
$(i)\Longrightarrow(ii)$ as claimed.

If the sequence $\{\sigma(B):B\in\mathcal{H}\}$ containes a subsequence
$\sigma(B_{k})\rightarrow-\infty$ then it may well happen that the set
$Spec(H_{-})=\overline{\mathfrak{S}_{\mathcal{H}}\mathfrak{+R}_{V}}$ will
contain a number of accumulating points, i.e. $Spec(H)$ in this case is not
discrete. In particular, $(ii)\Longrightarrow(i)$ if and only if the set
$\mathfrak{S}_{\mathcal{H}}\mathfrak{+R}_{V}$ has no accumulating points.
\end{proof}

\section{Rank one perturbations}

In this section we assume that the homogeneous ultrametric measure space
$(X,d,m)$ is \emph{countably infinite}. In this case $X$ can be identified
with a countable Abelian group $G$ equipped with an increasing sequence
$\{G_{n}\}_{n\in\mathbb{N}}$ of finite subgroups such that $\cap G_{n}=\{0\}$
and $\cup G_{n}=G$. Each ball in this ultrametric space is a set of the form
$g+G_{n}$ for some $g$ and $n$. As an example one can consider the group
$G=\mathbb{Z}(p_{1})\oplus$ $\mathbb{Z}(p_{2})\oplus\ldots$, weak sum of
cyclic groups, equipped with the sequence of its subgroups $G_{n}%
\simeq\mathbb{Z}(p_{1})\oplus\mathbb{Z}(p_{2})\oplus\ldots\oplus
\mathbb{Z}(p_{n})$.

Let $L$ be a homogeneous hierarchical Laplacian. We study spectral properties
of the Schr\"{o}dinger type operator $H=L+V$ with potential $V(x)=-\sigma
\delta_{a}(x)$, $\sigma>0$. Clearly $H$ can be written in the form%
\[
Hf(x)=Lf(x)-\sigma(f,\delta_{a})\delta_{a}(x),
\]
that is, $H$ can be regarded as a rank one perturbation of the operator $L$.
In this connection let us recall an abstract form of the Simon-Wolff theorem
\cite[Theorems 2 and 2']{SimonWolff} about pure point spectrum of rank one perturbations.

\paragraph{The Simon-Wolff criterion}

Let $A$ be a self-adjoint operator with simple spectrum on a Hilbert space
$\mathcal{H}$, and let $\varphi$ be a cyclic vector for $A$, that is,
$\{(A-\lambda)^{-1}\varphi$ $|$ $\mathfrak{\operatorname{Im}\lambda>}0\}$ is a
total set for $\mathcal{H}$. By the spectral theorem, $\mathcal{H}$ is unitary
equivalent to $L^{2}(\mathbb{R},\mu_{0})$ in such a way that $A$ is
multiplication by $x$ with cyclic vector $\varphi\equiv1$. Here $\mu_{0}$ is
the spectral measure of $\varphi$ for $A$. Let $H=A+\sigma(\varphi
,\cdot)\varphi$ be a rank one perturbation of the operator $A$. Set%
\[
F(x):=\int(x-y)^{-2}d\mu_{0}(y)=\lim_{\epsilon\rightarrow0}\left\Vert
(A-(x+i\epsilon)\mathrm{I})^{-1}\varphi\right\Vert ^{2}.
\]

\begin{theorem}
\label{SiWo}Fix an open interval $]a,b[$. The following are equivalent:

$(i)$ For a.e. $\sigma$, $H$ has only pure point spectrum in $]a,b[$.

$(ii)$ For a.e. $x\in]a,b[$, $F(x)<\infty$.
\end{theorem}

In general, if $\mathcal{H}_{0}$ is the closed subspace generated by vectors
$\{(A-\lambda\mathrm{I})^{-1}\varphi$ $|\mathfrak{\operatorname{Im}\lambda
>}0\}$, then its orthogonal complement $(\mathcal{H}_{0})^{\bot}$ is an
invariant space for $H$ and $H=A$ on $(\mathcal{H}_{0})^{\bot}$. Thus, the
extension from the cyclic to general case is clear.

The function $\varphi=\delta_{a}$ is not a cyclic vector for $L$ because the
operator $L$ has many compactly supported eigenfunctions $\phi$ having support
outside of $a$. Indeed, for any such $\phi$, for all $\mathfrak{\lambda}\in%
\mathbb{C}
$ with $\mathfrak{\operatorname{Im}\lambda>}0$ and for some $k$ we will have%
\[
((L-\lambda\mathrm{I})^{-1}\delta_{a},\phi)=(\delta_{a},(L-\overline{\lambda
}\mathrm{I})^{-1}\phi)=(\delta_{a},(\lambda_{k}-\overline{\lambda})^{-1}%
\phi)=0.
\]
We use the Krein type identity below to show that the spectrum of the operator
$H=L-\sigma\delta_{a}$ is pure point for all $\sigma$. Let $\psi
(x)=\mathcal{R}(\lambda,x,y)$ be the solution of the equation%
\[
L\psi(x)-\lambda\psi(x)=\delta_{y}(x).
\]
Let $\psi_{V}(x)=\mathcal{R}_{V}(\lambda,x,y)$ be the solution of the equation%
\[
H\psi_{V}(x)-\lambda\psi_{V}(x)=\delta_{y}(x).
\]
Notice that $L$ and $H$ are symmetric operators whence both $(x,y)\rightarrow
\mathcal{R}(\lambda,x,y)$ and \ $(x,y)\rightarrow\mathcal{R}_{V}(\lambda,x,y)$
are symmetric functions.

\begin{theorem}
\label{Krein identity I}In the notation introduced above
\begin{equation}
\mathcal{R}_{V}(\lambda,x,y)=\mathcal{R}(\lambda,x,y)+\frac{\sigma
\mathcal{R}(\lambda,x,a)\mathcal{R}(\lambda,a,y)}{1-\sigma\mathcal{R}%
(\lambda,a,a)}, \label{Krein}%
\end{equation}%
\begin{equation}
\mathcal{R}_{V}(\lambda,a,y)=\frac{\mathcal{R}(\lambda,a,y)}{1-\sigma
\mathcal{R}(\lambda,a,a)} \label{Krein I}%
\end{equation}
and%
\begin{equation}
\mathcal{R}_{V}(\lambda,a,a)=\frac{\mathcal{R}(\lambda,a,a)}{1-\sigma
\mathcal{R}(\lambda,a,a)}. \label{Krein I'}%
\end{equation}

\end{theorem}

\begin{proof}
We have%
\begin{align*}
L\psi_{V}(x)-\lambda\psi_{V}(x)  &  =\delta_{y}(x)+\sigma\delta_{a}(x)\psi
_{V}(x)\\
&  =\delta_{y}(x)+\sigma\delta_{a}(x)\psi_{V}(a).
\end{align*}
It follows that%
\begin{equation}
\psi_{V}(x)=\mathcal{R}(\lambda,x,y)+\sigma\psi_{V}(a)\mathcal{R}%
(\lambda,x,a). \label{Krein I''}%
\end{equation}
Setting $x=a$ in the above equation we obtain%
\[
\psi_{V}(a)=\mathcal{R}(\lambda,a,y)+\sigma\psi_{V}(a)\mathcal{R}%
(\lambda,a,a)
\]
or%
\[
\psi_{V}(a)(1-\sigma\mathcal{R}(\lambda,a,a))=\mathcal{R}(\lambda,a,y).
\]
Since $\psi_{V}(a)=\mathcal{R}_{V}(\lambda,a,y)$ we obtain equation
(\ref{Krein I'}). In turn, equations (\ref{Krein I'}) and (\ref{Krein I''})
imply (\ref{Krein}) and (\ref{Krein I}).
\end{proof}

\begin{theorem}
\label{Spectrum I} The operator $H=L$\ $-\sigma\delta_{a}$ has a pure point
spectrum which consists of at most one negative eigenvalue and countably many
positive eigenvalues with accumulating point $0$. $\ $

The operator $H$ has precisely one negative eigenvalue $\lambda_{-}^{\sigma}$
if and only if $\sigma>0$ and one of the following two conditions holds: $(i)$
the semigroup $(e^{-tL})_{t>0}$ is recurrent, $(ii)$ the semigroup
$(e^{-tL})_{t>0}$ is transient and $\mathcal{R}(0,a,a)>1/\sigma$. If it is the
case, then $Spec(H)$ consists of numbers%
\[
\lambda_{-}^{\sigma}<0<...<\lambda_{k+1}<\lambda_{k}^{\sigma}<\lambda
_{k}<...<\lambda_{2}<\lambda_{1}^{\sigma}<\lambda_{1}.
\]
Otherwise $Spec(H)$ consists of numbers
\[
0<...<\lambda_{k+1}<\lambda_{k}^{\sigma}<\lambda_{k}<...<\lambda_{2}%
<\lambda_{1}^{\sigma}<\lambda_{1}.
\]
If $\sigma<0$, then $Spec(H)$ consists of numbers
\[
0<...<\lambda_{k+1}<\lambda_{k}^{\sigma}<\lambda_{k}<...<\lambda_{2}%
<\lambda_{1}^{\sigma}<\lambda_{1}<\lambda_{+}^{\sigma}.
\]
The eigenvalues $\lambda_{k}$ are at the same time eigenvalues of the operator
$L$. All $\lambda_{k}$ have infinite multiplicity and compactly supported
eigenfunctions, the eigenfunctions of the operator $L$, whose supports do not
contain $a$.

The eigenvalue $\lambda_{k}^{\sigma}$ (resp. $\lambda_{-}^{\sigma}$,
$\lambda_{+}^{\sigma}$) is the unique solution of the equation%
\[
\mathcal{R}(\lambda,a,a)=1/\sigma
\]
in the interval $\left]  \lambda_{k+1},\lambda_{k}\right[  $ (resp.
$]-\infty,0[$, $]\lambda_{1},+\infty\lbrack$). Each $\lambda_{k}^{\sigma}$
(resp. $\lambda_{-}^{\sigma}$, $\lambda_{+}^{\sigma}$) has multiplicity one
and non-compactly supported eigenfunction $\psi_{k}(x)=\mathcal{R}(\lambda
_{k}^{\sigma},x,a)$ (resp. $\psi_{-}(x)=\mathcal{R}(\lambda_{-}^{\sigma}%
,x,a)$, $\psi_{+}(x)=\mathcal{R}(\lambda_{+}^{\sigma},x,a)$).
\end{theorem}

\textbf{%
\begin{figure}
[tbh]
\begin{center}
\includegraphics[
height=2.6307in,
width=5.1831in
]%
{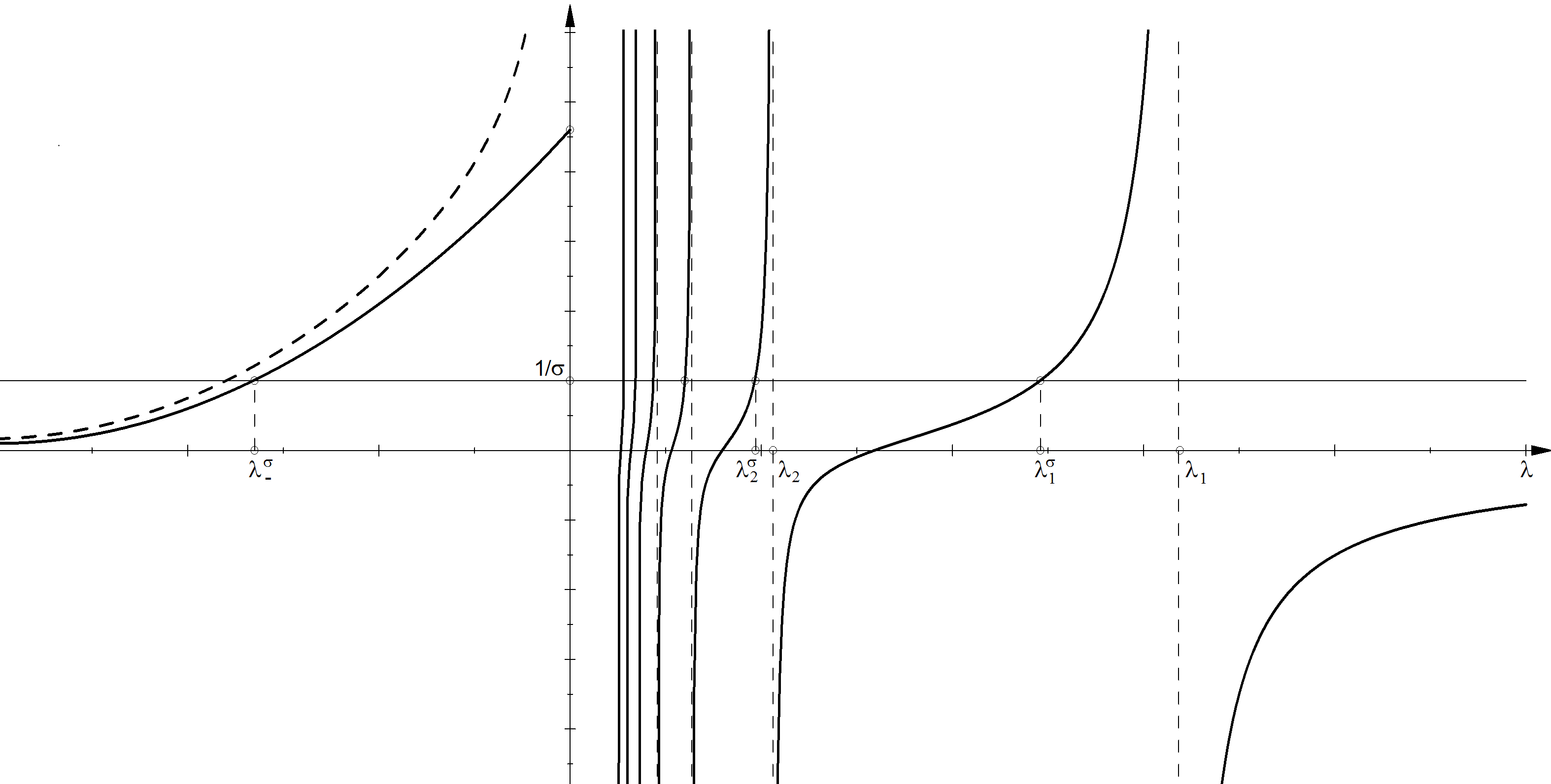}%
\caption{The roots $\left\{  \lambda_{\ast}^{\sigma}\right\}  $ of the
equation $\mathcal{R}\left(  \lambda,a,a\right)  =1/\sigma.$ The dashed graph
corresponds to a recurrent case, the solid graph -- to the transient case.}%
\label{pic2}%
\end{center}
\end{figure}
}

\begin{proof}
Let $\Upsilon(X)$ be the tree of balls associated with the ultrametric space
$(X,d)$.\ Consider in $\Upsilon(X)$ the infinite geodesic path from $a$ to
$\varpi:$\ $\{a\}=B_{0}\varsubsetneq B_{1}\varsubsetneq...\varsubsetneq
B_{k}\varsubsetneq...$ . The series below converges uniformly and in $L^{2},$%
\begin{equation}
\delta_{a}=\left(  \frac{1_{B_{0}}}{m(B_{0})}-\frac{1_{B_{1}}}{m(B_{1}%
)}\right)  +\left(  \frac{1_{B_{1}}}{m(B_{1})}-\frac{1_{B_{2}}}{m(B_{2}%
)}\right)  +...=\sum_{k=0}^{\infty}f_{B_{k}}. \label{delta_a decomposition}%
\end{equation}
Notice that all $f_{B_{k}}$ are eigenfunctions of the operator $L$, i.e.
$Lf_{B_{k}}=\lambda(B_{k+1})f_{B_{k}}=\lambda_{k+1}f_{B_{k}}.$ By definition
$\mathcal{R}(\lambda,x,y)=(L-\lambda)^{-1}\delta_{y}(x)$ whence we obtain
\begin{align*}
\mathcal{R}(\lambda,a,a)  &  =\frac{1}{\lambda_{1}-\lambda}f_{B_{0}}%
(a)+\frac{1}{\lambda_{2}-\lambda}f_{B_{1}}(a)+...\\
&  =\frac{1}{\lambda_{1}-\lambda}\left(  \frac{1}{m(B_{0})}-\frac{1}{m(B_{1}%
)}\right) \\
&  +\frac{1}{\lambda_{2}-\lambda}\left(  \frac{1}{m(B_{1})}-\frac{1}{m(B_{2}%
)}\right)  +...\text{ ,}%
\end{align*}
or in the final form%
\begin{equation}
\mathcal{R}(\lambda,a,a)=%
{\displaystyle\sum\limits_{k=1}^{\infty}}
\frac{A_{k}}{\lambda_{k}-\lambda},\text{ \ }A_{k}=\left(  \frac{1}{m(B_{k-1}%
)}-\frac{1}{m(B_{k})}\right)  . \label{Resolvent poles}%
\end{equation}
Since $\lambda\rightarrow\mathcal{R}(\lambda,a,a)$ is an increasing function,
the equation%
\begin{equation}
1-\sigma\mathcal{R}(\lambda,a,a)=0,\text{ \ }\sigma\neq0,
\label{sigma-equation}%
\end{equation}
has precisely one solution $\lambda_{k}^{\sigma}$ lying in each open interval
$\left]  \lambda_{k+1},\lambda_{k}\right[  $ ,%
\[
\lambda_{k+1}<\lambda_{k}^{\sigma}<\lambda_{k}\text{, \ }k=1,2,...\text{ .}%
\]
\textbf{Claim 1} All numbers $\lambda_{k}^{\sigma}$ are eigenvalues of the
operator $H$. Indeed, the function $\psi(x)=\mathcal{R}(\lambda,x,a)$ with
$\lambda=$ $\lambda_{k}^{\sigma}$ satisfies the equation%
\begin{align*}
H\psi(x)-\lambda\psi(x)  &  =L\psi(x)-\lambda\psi(x)-\sigma\delta_{a}%
(x)\psi(x)\\
&  =L\psi(x)-\lambda\psi(x)-\sigma\delta_{a}(x)\psi(a)\\
&  =L\psi(x)-\lambda\psi(x)-\delta_{a}(x)=0.
\end{align*}
\textbf{Claim 2 }All numbers $\lambda_{k}$ are eigenvalues of the operator
$H$. Indeed, for any ball $B$ which does not contain $a$ but belongs to the
horocycle $\mathcal{H}_{k-1}$ we have
\[
Hf_{B}=Lf_{B}=\lambda_{k}f_{B}.
\]
When $\sigma>0$ there may exist one more eigenvalue $\lambda_{-}^{\sigma}<0,$
a solution of the equation (\ref{sigma-equation}). Indeed, $\lambda
\rightarrow\mathcal{R}(\lambda,a,a)$ is an increasing function, continuous on
the interval $]-\infty,0]$. Since $\mathcal{R}(\lambda,a,a)\rightarrow0$ as
$\lambda\rightarrow-\infty$ and $\mathcal{R}(\lambda,a,a)\rightarrow
\mathcal{R}(0,a,a)\leq+\infty$ as $\lambda\rightarrow-0$, equation
(\ref{sigma-equation}) has unique solution $\lambda=\lambda_{-}^{\sigma}<0$ in
the cases $(i)$ and $(ii).$The proof of the theorem is finished.
\end{proof}

\begin{example}
\label{Example "One-point pert."}The Dyson's Laplacian. Consider the set
$X\mathcal{=}\{0,1,2,...\}$ equipped with the counting measure $m$ and with
the set of partitions $\{\Pi_{r}:$ $r=0,1,...\}$ each of which consists of all
rank $r$ intervals $I_{r}=\{x\in\mathcal{X}:$ $kp^{r}\leq x<(k+1)p^{r}\}$. The
set of partitions $\{\Pi_{r}\}$ generates the ultrametric structure on $X$ and
the hierarchical Laplacian
\[
\mathrm{D}^{\alpha}f(x)=%
{\displaystyle\sum\limits_{r=1}^{+\infty}}
(1-\kappa)\kappa^{r-1}\left(  f(x)-\frac{1}{m(I_{r}(x))}%
{\displaystyle\int\limits_{I_{r}(x)}}
fdm\right)  ,\text{ }\kappa=p^{-\alpha},
\]
where the sum is taken over all rank $r$ intervals $I_{r}(x)$ which contain
$x$.

The operator $\mathrm{D}^{\alpha}$ admits a complete system of compactly
supported eigenfunctions. Indeed, let $I$ be an interval of rank $r$, and
$I_{1},I_{2},...,I_{p}$ be its subintervals of rank $r-1$. Let us consider $p$
functions%
\[
f_{I_{i}}=\frac{1_{I_{i}}}{m(I_{i})}-\frac{1_{I}}{m(I)},\text{ }i=1,2,...,p.
\]
Each function $f_{I_{i}}$\ belongs to the domain of the operator
$\mathrm{D}^{\alpha}$\ and%
\[
\mathrm{D}^{\alpha}f_{I_{i}}=\kappa^{r-1}f_{I_{i}}\text{. \ }%
\]
When $I$ runs over the set all $p$-adic intervals the set of eigenfunctions
$f_{I_{i}}$ forms a complete system in $L^{2}(X,m)$. In particular,
$\mathrm{D}^{\alpha}$ is essentially self-adjoint operator having pure point
spectrum
\[
Spec(\mathrm{D}^{\alpha})=\{0\}\cup\{\kappa^{r-1}:r\in\mathbb{N}\}.
\]
Clearly each eigenvalue $\lambda(I)=\kappa^{r-1}$ has infinite multiplicity.
Let us compute the value $\mathcal{R}(\lambda):=\mathcal{R}(\lambda,0,0)$ of
the resolvent kernel for $\mathrm{D}^{\alpha}$. By equation
(\ref{Resolvent poles}), we have
\[
\mathcal{R}(\lambda)=%
{\displaystyle\sum\limits_{k\geq1}}
\frac{A_{k}}{\lambda_{k}-\lambda}=(p-1)%
{\displaystyle\sum\limits_{k\geq1}}
\frac{1}{p^{k}(\lambda_{k}-\lambda)}.
\]
In particular, $\mathcal{R}(0)=+\infty$ if and only if $\alpha\geq1$,
otherwise
\[
\mathcal{R}(0)=\frac{p-1}{p}%
{\displaystyle\sum\limits_{k\geq0}}
\frac{1}{p^{k(1-\alpha)}}=\frac{p-1}{p-p^{\alpha}}.
\]
Consider the operator $H=\mathrm{D}^{\alpha}-\sigma\delta_{0}$, $\sigma>0$.
Let us compute the number $Neg(H)$ of negative eigenvalues of the operator $H$
counted with their multiplicity. By Theorem \ref{Spectrum I}, the operator $H$
has at most one negative eigenvalue. It has exactly one negative eigenvalue if
and only if either $\alpha\geq1$ or $0<\alpha<1$ and $\sigma>(p-p^{\alpha
})(p-1)^{-1}$. If we denote the set of pairs $(\alpha,\sigma)$ which satisfy
the above conditions by $\Omega$ and by $\Omega_{0}=\mathbb{R}_{+}%
^{2}\setminus\Omega$ its complement, we obtain%
\[
Neg(H)=\left\{
\begin{array}
[c]{cc}%
1 & \text{if }(\alpha,\sigma)\in\Omega\\
0 & \text{if }(\alpha,\sigma)\in\Omega_{0}%
\end{array}
\right.
\]
which is shown on the picture below.%
\begin{figure}
[tbh]
\begin{center}
\includegraphics[
height=3.054in,
width=5.4121in
]%
{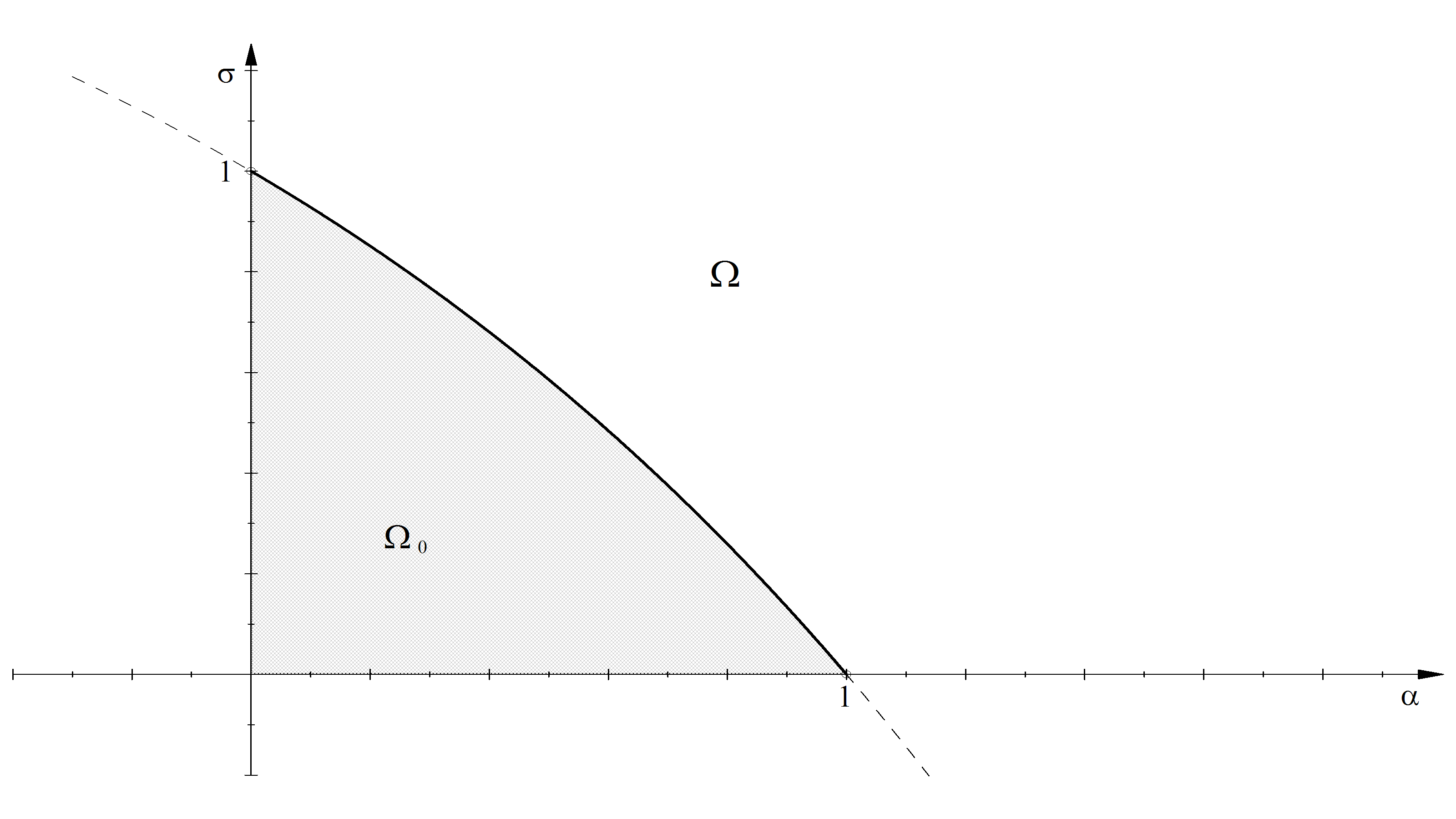}%
\caption{Sets $\Omega_{0}$ and $\Omega$}%
\label{pic3}%
\end{center}
\end{figure}

\end{example}

\section{Finite rank perturbations}

As in the previous section the ultrametric measure space $(X,d,m)$ is
\emph{countably infinite and homogeneous. For convenience, we assume that}
$m(B)=\mathrm{diam}(B)$ \emph{for any non-singleton ball} $B.$

Let $L$ be a homogeneous hierarchical Laplacian. We study spectral properties
of the Schr\"{o}dinger type operator $H=L+V$ with potential $V(x)=-\sum
_{i=1}^{N}\sigma_{i}\delta_{a_{i}}(x)$, $\sigma_{i}>0$. Clearly $H$ can be
written in the form%
\[
Hf(x)=Lf(x)-\sum_{i=1}^{N}\sigma_{i}(f,\delta_{a_{i}})\delta_{a_{i}}(x),
\]
that is, $H$ can be regarded as rank $N$ perturbation of the operator $L$.
Throughout this section we use the following notation

\begin{itemize}
\item $\mathcal{R}(\lambda,x,y)$ is the solution of the equation
$L\psi(x)-\lambda\psi(x)=\delta_{y}(x)$. We set $\mathcal{R}(\lambda
,x,\overrightarrow{a}):=(\mathcal{R}(\lambda,x,a_{i}))_{i=1}^{N}$, and
$\mathcal{R}(\lambda,\overrightarrow{a},\overrightarrow{a}):=(\mathcal{R}%
(\lambda,a_{j},a_{i}))_{i,j=1}^{N}$.

\item $\mathcal{R}_{V}(\lambda,x,y)$ is the solution of the equation
$H\psi(x)-\lambda\psi(x)=\delta_{y}(x)$. We set $\mathcal{R}_{V}%
(\lambda,x,\overrightarrow{a}):=(\mathcal{R}_{V}(\lambda,x,a_{i}))_{i=1}^{N}$,
and $\mathcal{R}_{V}(\lambda,\overrightarrow{a},\overrightarrow{a}%
):=(\mathcal{R}_{V}(\lambda,a_{j},a_{i}))_{i,j=1}^{N}.$

\item $\Sigma:=\mathrm{diag}(\sigma_{i}:i=1,...,N).$
\end{itemize}

\begin{theorem}
\label{pure point spectrum}The following properties hold true:

\begin{description}
\item[1.] The set $Spec(H)$ is pure point, its essential part $Spec_{ess}(H)$
coincides with the set $Spec(L)=\{0\}\cup\{\lambda_{k}\}$, its discrete part
$Spec_{d}(H)$ in each open interval lying in the complement of $Spec(L)$
consists of at most $N$ distinct points, solutions of the equation
\begin{equation}
\det(\Sigma^{-1}-\mathcal{R}(\lambda,\overrightarrow{a},\overrightarrow
{a}))=0. \label{discrete spectrum}%
\end{equation}

\item[2.] For each $k\in\mathbb{N}$ there exists $\delta>1$ such that
$\min_{i\neq j}d(a_{i},a_{j})>\delta$ implies that the operator $H$ has
precisely $N$ distinct eigenvalues in each open interval $(\lambda
_{s+1},\lambda_{s})$: $1\leq s\leq k$. Moreover, there exists precisely $N$
distinct negative eigenvalues of the operator $H$ provided one of the
following two conditions is satisfied:

$(2.1)$ The semigroup $(e^{-tL})_{t>0}$ is recurrent.

$(2.2)$ The semigroup $(e^{-tL})_{t>0}$ is transient and all $1/\sigma
_{i}<\mathcal{R}(0,a,a)$. \footnote{Thanks to the homogenuity assumption
$\mathcal{R}(\lambda,a,a)$ does not depend on $a$}
\end{description}
\end{theorem}

The proof of the first part of Theorem \ref{pure point spectrum} is based on
the Weyl's theorem on the essential spectrum of compactly perturbed symmetric
operators, see \cite[Theorem IV.5.35]{Kato}, and on the following lemma.

\begin{lemma}
\label{Weyl perturbation}Let $A$ and $B$ be two symmetric bounded operators
and $H=A+B$. Assume that $B$ is of rank $N$ operator. Let $(a,b)$ be an
interval lying in the complement of the set $Spec(A)$. Then the set
$Spec(H)\cap(a,b)$ consists of at most $N$ distinct points.
\end{lemma}

\begin{proof}
By the Weyl's essential spectrum theorem $Spec_{ess}(H)$ coincides with the
set $Spec_{ess}(L)=\{0\}\cup\{\lambda_{k}\}$. Hence the set $Spec(H)\cap(a,b)$
may contain only finite number of eigenvalues each of which has finite
multiplicity. Consider the case $N=1$, that is, the operator $B$ is of the
form%
\[
Bf=\sigma_{1}(f,f_{1})f_{1}.
\]
Let $\lambda\in(a,b)$ and let $f$ be a non-trivial solution of the equation
$Hf-\lambda f=0$. Then $f$ can be written in the form
\begin{equation}
f=-\sigma_{1}(f,f_{1})R_{\lambda}f_{1} \label{sigma-eigenfunction}%
\end{equation}
where $R_{\lambda}=(A-\lambda)^{-1}$ is the resolvent operator. It follows
that $(f,f_{1})\neq0$ and
\[
(f,f_{1})=-\sigma_{1}(f,f_{1})(R_{\lambda}f_{1},f_{1}),
\]
or
\begin{equation}
\sigma_{1}(R_{\lambda}f_{1},f_{1})+1=0. \label{sigma eq.}%
\end{equation}
The function $\phi(\lambda)=(R_{\lambda}f_{1},f_{1})$ is strictly increasing
on the interval $(a,b)$. Indeed, applying the resolvent identity we get%
\[
\frac{d\phi(\lambda)}{d\lambda}=(R_{\lambda}^{2}f_{1},f_{1})=\left\Vert
R_{\lambda}f_{1}\right\Vert ^{2}>0.
\]
It follows that equation (\ref{sigma eq.}) has at most one solution lying in
the interval $(a,b)$. Assume that equation (\ref{sigma eq.}) has a solution,
denote it $\lambda_{\ast}$. Then (\ref{sigma-eigenfunction}) implies that the
vector $f_{\ast}:=R_{\lambda_{\ast}}f_{1}/\left\Vert R_{\lambda_{\ast}}%
f_{1}\right\Vert $ satisfies the equation%
\[
Hf_{\ast}-\lambda_{\ast}f_{\ast}=0.
\]
Thus the operator $H$ has at most one eigenvalue in the interval $(a,b)$.

Without loss of generality we may provide the induction from $N=1$ to $N=2$.
Thus assuming that the perturbation operator $B$ is of the form%
\[
Bf=\sigma_{1}(f,f_{1})f_{1}+\sigma_{2}(f,f_{2})f_{2}%
\]
we set
\[
A^{\prime}f:=Af+\sigma_{1}(f,f_{1})f_{1}%
\]
and%
\[
Hf:=A^{\prime}f+\sigma_{2}(f,f_{2})f_{2}.
\]
Observe that the operator $A^{\prime}$ may have in the interval $(a,b)$ at
most one eigenvalue $\lambda_{\ast}$. The corresponding eigenspace is
one-dimensional, call it $\left\langle f_{\ast}\right\rangle $, where
$f_{\ast}:=R_{\lambda_{\ast}}f_{1}/\left\Vert R_{\lambda_{\ast}}%
f_{1}\right\Vert $. Let us consider two cases.

\underline{First case}: Assume that $f_{2}\perp f_{\ast}$. Then $Hf_{\ast
}=A^{\prime}f_{\ast}=\lambda_{\ast}f_{\ast}$, i.e. $\lambda_{\ast}$ is an
eigenvalue of the operator $H$. It follows that the orthogonal complement
$\left\langle f_{\ast}\right\rangle ^{\perp}$ is a joint invariant subspace of
the operators $H$ and $A^{\prime}$ and that these operators being restricted
to $\left\langle f_{\ast}\right\rangle ^{\perp}$, call them $H_{\perp}$ and
$A_{\perp}^{\prime}$, satisfy
\[
H_{\perp}f=A_{\perp}^{\prime}f+\sigma_{2}(f,f_{2})f_{2}.
\]
The operator $A_{\perp}^{\prime}$ has no eigenvalues in the interval $(a,b)$.
Hence, by what we have already shown in the first part of the proof, the
operator $H_{\perp}$ has at most one eigenvalue in the interval $(a,b)$. It
follows that the operator $H$ has at most two eigenvalues in the interval
$(a,b)$.

\underline{Second case:} Assume that $f_{2}$ and $f_{\ast}$ are not
orthogonal. Let $\mathcal{R}_{\lambda}:=(H-\lambda\mathrm{I})^{-1}$ and
$R_{\lambda}^{\prime}:=(A^{\prime}-\lambda\mathrm{I})^{-1}$ be the resolvent
operators. The following identity holds true%
\begin{equation}
(\mathcal{R}_{\lambda}f,g)=(R_{\lambda}^{\prime}f,g)-\frac{\sigma
_{2}(R_{\lambda}^{\prime}f,f_{2})(R_{\lambda}^{\prime}f_{2},g)}{1+\sigma
_{2}(R_{\lambda}^{\prime}f_{2},f_{2})} \label{resolvent eq}%
\end{equation}
for any $f,g$ and $\lambda\neq\lambda_{\ast}$ lying in $(a,b)$. Using the
spectral resolution formula for the operator $A^{\prime}$, the fact that its
spectral function $E_{\lambda}$ in $(a,b)$ has the only jump at $\lambda
=\lambda_{\ast}$ and that the value of the jump $\Delta E_{\lambda_{\ast}}$ is
the operator of orthogonal projection on the subspace $\left\langle f_{\ast
}\right\rangle $ we get%
\begin{equation}
(R_{\lambda}^{\prime}f,f)=\frac{(f_{\ast},f)^{2}}{\lambda-\lambda_{\ast}%
}+O_{1}(1) \label{asymptotic 1}%
\end{equation}
and%
\begin{equation}
(R_{\lambda}^{\prime}f,f_{2})=\frac{(f_{\ast},f)(f_{\ast},f_{2})}%
{\lambda-\lambda_{\ast}}+O_{2}(1) \label{asymptotic 2}%
\end{equation}
where $O_{i}(1)$ are analytic functions. Substituting asymptotic equations
(\ref{asymptotic 1}) and (\ref{asymptotic 2}) in equation (\ref{resolvent eq})
we get analyticity of the function $\lambda\rightarrow$ $(\mathcal{R}%
_{\lambda}f,f)$ at $\lambda=\lambda_{\ast}$. In particular, this shows that
$\lambda=\lambda_{\ast}$ is not an eigenvalue of $H$.

On the other hand $\lambda_{\ast}$ splits the interval $(a,b)$ in two parts
$(a,\lambda_{\ast})$ and $(\lambda_{\ast},b)$ each of which does not contain
eigenvalues of the operator $A^{\prime}$. Then, as we have already shown, each
of these intervals contains at most one eigenvalue of the operator $H$. Since
$\lambda_{\ast}$ is not an eigenvalue of the operator $H$, the number of
distinct eigenvalues of $H$ in the interval $(a,b)$ is at most two. The proof
of the lemma is finished.
\end{proof}

\bigskip

\textbf{Proof of Theorem} \ref{pure point spectrum} (second part): Let
$\lambda\in Spec_{d}(H)$ and let $\psi(x)$ be the corresponding eigenfunction,
i.e.
\[
H\psi(x)-\lambda\psi(x)=0.
\]
We have%
\[
L\psi(x)-\lambda\psi(x)=\sum_{i=1}^{N}\sigma_{i}\psi(a_{i})\delta_{a_{i}}(x)
\]
or applying to this equation the resolvent operator $(L-\lambda)^{-1}$ we get%
\begin{equation}
\psi(x)=\sum_{i=1}^{N}\sigma_{i}\psi(a_{i})\mathcal{R}(\lambda,x,a_{i}).
\label{eigenfunctions}%
\end{equation}
Taking consequently $x=a_{1},a_{2},...,a_{N}$ in equation
(\ref{eigenfunctions}) we obtain a homogeneous system of $N$ linear equations
with $N$ variables%
\begin{equation}
\psi(a_{j})=\sum_{i=1}^{N}\sigma_{i}\psi(a_{i})\mathcal{R}(\lambda,a_{j}%
,a_{i}) \label{System Eq'}%
\end{equation}
or in the vector form%
\begin{equation}
\Psi=\mathcal{R}(\lambda,\overrightarrow{a},\overrightarrow{a})\Sigma\Psi,
\label{Systen Eq}%
\end{equation}
where $\Psi=(\psi(a_{i}):i=1,...,N)$. The system (\ref{Systen Eq}) has a
non-trivial solution if and only if
\begin{equation}
\det(\Sigma^{-1}-\mathcal{R}(\lambda,\overrightarrow{a},\overrightarrow{a})=0.
\label{det-Eq}%
\end{equation}
Observe that the variable $z:=\mathcal{R}(\lambda,a_{i},a_{i})$ does not
depend on $a_{i},$ and its range is the whole interval $]-\infty,\infty
\lbrack$ when $\lambda$ takes values in each of open interval $]\lambda
_{k+1},\lambda_{k}[.$ Equation (\ref{det-Eq}) can be written as characteristic
equation
\begin{equation}
\det(\mathfrak{A}-z\mathrm{I})=0 \label{char. Eq}%
\end{equation}
where $\mathfrak{A}=(\mathfrak{a}_{ij})_{i,j=1}^{N}$ is symmetric $N\times N$
matrix with entries
\begin{equation}
\text{\ \ }\mathfrak{a}_{ij}=\left\{
\begin{array}
[c]{ccc}%
1/\sigma_{i}\  & \text{for} & i=j\\
-\mathcal{R}(\lambda,a_{i},a_{j}) & \text{for} & i\neq j
\end{array}
\right.  . \label{matrix A-fract}%
\end{equation}
Let us compute $\mathcal{R}(\lambda,a_{i},a_{j})$. For any two neighboring
balls $B\subset B^{\prime}$ let us denote
\[
A(B)=\frac{1}{m(B)}-\frac{1}{m(B^{\prime})}.
\]
\emph{Remember that we normalize} $m$ \emph{so that} $m(B)=\mathrm{diam}(B)$
\emph{for any non-singleton ball} $B$ whence for such $B$,%
\begin{equation}
A(B)=\frac{1}{\mathrm{diam}(B)}-\frac{1}{\mathrm{diam}(B^{\prime})}.
\label{A(B)}%
\end{equation}
Let $a_{i}\curlywedge a_{j}$ be the minimal ball which contains both $a_{i}$
and $a_{j}$. Following the same line of reasons as in the proof of equation
(\ref{Resolvent poles}) we obtain
\begin{equation}
\mathcal{R}(\lambda,a_{i},a_{i})=%
{\displaystyle\sum\limits_{B:\text{ }a_{i}\in B}}
\frac{A(B)}{\lambda(B)-\lambda}. \label{R on diag}%
\end{equation}
Similarly, for all $i\neq j$ we get
\begin{equation}
\mathcal{R}(\lambda,a_{i},a_{j})=-\frac{\mathrm{d}(a_{i},a_{j})^{-1}}%
{\lambda(a_{i}\curlywedge a_{j})-\lambda}+%
{\displaystyle\sum\limits_{B:\text{ }a_{i}\curlywedge a_{j}\subset B}}
\frac{A(B)}{\lambda(B)-\lambda}. \label{R outside diag}%
\end{equation}
Let $\lambda>\lambda(a_{i}\curlywedge a_{j})$. Equations (\ref{A(B)}),
(\ref{R outside diag}) and the fact $S\subset T\Rightarrow$ $\lambda
(S)>\lambda(T)$ imply that
\begin{align*}
\mathcal{R}(\lambda,a_{i},a_{j})  &  =\frac{\mathrm{d}(a_{i},a_{j})^{-1}%
}{\lambda-\lambda(a_{i}\curlywedge a_{j})}-%
{\displaystyle\sum\limits_{B:\text{ }a_{i}\curlywedge a_{j}\subset B}}
\frac{A(B)}{\lambda-\lambda(B)}\\
&  >\frac{\mathrm{d}(a_{i},a_{j})^{-1}}{\lambda-\lambda(a_{i}\curlywedge
a_{j})}-\frac{1}{\lambda-\lambda(a_{i}\curlywedge a_{j})}%
{\displaystyle\sum\limits_{B:\text{ }a_{i}\curlywedge a_{j}\subset B}}
A(B)\\
&  =\frac{1}{\lambda-\lambda(a_{i}\curlywedge a_{j})}\left(  \frac
{1}{\mathrm{d}(a_{i},a_{j})}-\frac{1}{\mathrm{diam}(a_{i}\curlywedge
a_{j})^{\prime}}\right)  >0.
\end{align*}
Hence for $\lambda>\lambda(a_{i}\curlywedge a_{j})$ we obtain
\begin{equation}
0<\mathcal{R}(\lambda,a_{i},a_{j})<\frac{\mathrm{d}(a_{i},a_{j})^{-1}}%
{\lambda-\lambda(a_{i}\curlywedge a_{j})}. \label{R-lambda-ineq}%
\end{equation}
Notice that $\lambda(B)\rightarrow0$ as $\mathrm{diam}(B)\rightarrow\infty$.
Let us fix $k$ and let us consider $\lambda>\lambda_{k+1}$. Let us choose
$\delta>1$ such that if $\min_{i\neq j}d(a_{i},a_{j})\geq\delta$ then
$\lambda(a_{i}\curlywedge a_{j})<\lambda_{k}/2$. Then for all $i\neq j$ we get
$\lambda-\lambda(a_{i}\curlywedge a_{j})>\lambda_{k}/2$ and thus%
\begin{equation}
\left\vert \mathcal{R}(\lambda,a_{i},a_{j})\right\vert <\frac{2}{\delta
\lambda_{k}}:=\frac{\varepsilon(\delta)}{N}. \label{R-lambda inequality}%
\end{equation}
Let us increase if necessary $\delta$ so that the intervals
\[
\{s:\left\vert 1/\sigma_{i}-s\right\vert \leq\varepsilon(\delta)\},\text{
}i=1,2,...,N,
\]
do not intersect. By Gershgorin Circle Theorem the matrix $\mathfrak{A}$
admits $N$ different eigenvalues $\mathfrak{a}_{i}$ each of which lies in the
corresponding open interval
\[
\{s:\left\vert 1/\sigma_{i}-s\right\vert <\varepsilon(\delta)\},\text{
}i=1,2,...,N.
\]
The eigenvalues $\mathfrak{a}_{i},i=1,2,...,N,$ are analytic functions of
$\lambda$ in each open interval $(\lambda_{s+1},\lambda_{s})$, $1\leq s\leq
k$, see \cite[Theorem XII.1]{ReedSimon}. Whence in each interval
$(\lambda_{s+1},\lambda_{s})$ the number of different solutions of the
equations $\mathfrak{a}_{i}=\mathcal{R}(\lambda,a_{i},a_{i})$ is at least $N$.
By Lemma \ref{Weyl perturbation} the number of different solutions is at most
$N$. Thus the number of different solutions is precisely $N$ as claimed.

\begin{theorem}
\label{E-f for discr.spec.}The set $Spec_{d}(H)$ coincides with the set of
solutions of equation (\ref{discrete spectrum}). Each eigenfunction
$\psi_{\lambda}(x)$ corresponding to $\lambda\in Spec_{d}(H)$ can be
represented as linear combination of functions $\mathcal{R}(\lambda,x,a_{i})$,
that is,%
\[
\psi_{\lambda}(x)=\sum_{i=1}^{N}\zeta_{i}\mathcal{R}(\lambda,x,a_{i}).
\]
Thus, support of $\psi_{\lambda}$ is the whole space $X$ whereas the
eigenfunctions $f_{B}$ corresponding to the eigenvalues $\lambda(B)\in
Spec_{ess}(H)$ are compactly supported.
\end{theorem}

\begin{proof}
The proof is straightforward: we apply equations (\ref{eigenfunctions}) and
(\ref{System Eq'}) to get the result, see the first part of the proof of
Theorem \ref{pure point spectrum} (second statement).
\end{proof}

\begin{theorem}
\label{perturbed resolvent}For $\lambda\notin Spec(H)$ the following
identities hold true:
\begin{equation}
\mathcal{R}_{V}(\lambda,x,y)=\mathcal{R}(\lambda,x,y)+\mathcal{R}%
(\lambda,x,\overrightarrow{a})(\Sigma^{-1}-\mathcal{R}(\lambda,\overrightarrow
{a},\overrightarrow{a}))^{-1}\mathcal{R}(\lambda,\overrightarrow
{a},y),\footnote{For a matrix $A$ and vectors $\xi$ and $\eta$ we write $\xi
A\eta:=\sum_{i,j}a_{ij}\xi_{i}\eta_{j}.$ } \label{R_V}%
\end{equation}%
\begin{equation}
\Sigma\mathcal{R}_{V}(\lambda,\overrightarrow{a},y)=(\Sigma^{-1}%
-\mathcal{R}(\lambda,\overrightarrow{a},\overrightarrow{a}))^{-1}%
\mathcal{R}(\lambda,\overrightarrow{a},y) \label{R_V1}%
\end{equation}
and%
\begin{equation}
\Sigma\mathcal{R}_{V}(\lambda,\overrightarrow{a},\overrightarrow{a}%
)=(\Sigma^{-1}-\mathcal{R}(\lambda,\overrightarrow{a},\overrightarrow
{a}))^{-1}\mathcal{R}(\lambda,\overrightarrow{a},\overrightarrow{a}).
\label{R_V2}%
\end{equation}
In particular, the operator $T(\lambda):=(H-\lambda\mathrm{I})^{-1}%
-(L-\lambda\mathrm{I})^{-1}$ is of finite rank $N$. Its operator norm can be
estimated as follows
\begin{equation}
\left\Vert T(\lambda)\right\Vert \leq\left\Vert (\Sigma^{-1}-\mathcal{R}%
(\lambda,\overrightarrow{a},\overrightarrow{a}))^{-1}\right\Vert \left\Vert
(L-\lambda\mathrm{I})^{-1}\right\Vert ^{2}. \label{R_V3}%
\end{equation}

\end{theorem}

\begin{proof}
Recall that $Spec(H)$ coincides with the union of two sets: $Spec(L)$ and the
set of those $\lambda\in\mathbb{R}$ for which $\det(\Sigma^{-1}-\mathcal{R}%
(\lambda,\overrightarrow{a},\overrightarrow{a}))=0$. The proof of the theorem
is similar to its one-dimensional version Theorem \ref{Krein identity I}.
Clearly we can write the following equation
\begin{align*}
L\mathcal{R}_{V}(\lambda,x,y)-\lambda\mathcal{R}_{V}(\lambda,x,y)  &
=\delta_{y}(x)+\sum_{i=1}^{N}\sigma_{j}\delta_{a_{j}}(x)\mathcal{R}%
_{V}(\lambda,x,y)\\
&  =\delta_{y}(x)+\sum_{j=1}^{N}\sigma_{j}\mathcal{R}_{V}(\lambda
,a_{j},y)\delta_{a_{j}}(x),
\end{align*}
or equivalently%
\begin{equation}
\mathcal{R}_{V}(\lambda,x,y)=\mathcal{R}(\lambda,x,y)+\sum_{j=1}^{N}\sigma
_{j}\mathcal{R}_{V}(\lambda,a_{j},y)\mathcal{R}(\lambda,x,a_{j}). \label{RV1}%
\end{equation}
Substituting consequently $x=a_{1},a_{2},...,a_{N}$ we obtain system of $N$
linear equations with $N$ variables%
\[
\mathcal{R}_{V}(\lambda,a_{i},y)=\mathcal{R}(\lambda,a_{i},y)+\sum_{j=1}%
^{N}\sigma_{j}\mathcal{R}(\lambda,a_{i},a_{j})\mathcal{R}_{V}(\lambda
,a_{j},y)
\]
or in the vector form%
\begin{equation}
(\mathrm{I}-\mathcal{R}(\lambda,\overrightarrow{a},\overrightarrow{a}%
)\Sigma)\mathcal{R}_{V}(\lambda,\overrightarrow{a},y)=\mathcal{R}%
(\lambda,\overrightarrow{a},y). \label{RV2}%
\end{equation}
Assuming that $\lambda\notin Spec(H)$, in particular $\det(\mathrm{I}%
-\mathcal{R}(\lambda,\overrightarrow{a},\overrightarrow{a})\Sigma)\neq0$, we
get
\begin{equation}
\mathcal{R}_{V}(\lambda,\overrightarrow{a},y)=(\mathrm{I}-\mathcal{R}%
(\lambda,\overrightarrow{a},\overrightarrow{a})\Sigma)^{-1}\mathcal{R}%
(\lambda,\overrightarrow{a},y) \label{RV3}%
\end{equation}
Evidently equations (\ref{RV1}) and (\ref{RV3}) imply equations (\ref{R_V}),
(\ref{R_V1}) and (\ref{R_V2}).

Equation $T(\lambda)=(H-\lambda\mathrm{I})^{-1}(L-H)(L-\lambda\mathrm{I}%
)^{-1}$ applies that $T(\lambda))$ is of rank $N$. Finally, equation
(\ref{R_V3}) follows from equation (\ref{R_V}). Indeed, for $f\in L^{2}(X,m)$
we introduce (finite-dimensional) vectors $\mathcal{R}(\lambda
,f,\overrightarrow{a}):=\sum_{x}f(x)\mathcal{R}(\lambda,x,\overrightarrow{a})$
and $\mathcal{R}(\lambda,\overrightarrow{a},f):=\sum_{y}f(x)\mathcal{R}%
(\lambda,\overrightarrow{a},y)$, then
\begin{align*}
(T(\lambda)f,f)  &  =\sum_{x,y}f(x)\mathcal{R}(\lambda,x,\overrightarrow
{a})(\Sigma^{-1}-\mathcal{R}(\lambda,\overrightarrow{a},\overrightarrow
{a}))^{-1}\mathcal{R}(\lambda,\overrightarrow{a},y)f(y)\\
&  =\mathcal{R}(\lambda,f,\overrightarrow{a})(\Sigma^{-1}-\mathcal{R}%
(\lambda,\overrightarrow{a},\overrightarrow{a}))^{-1}\mathcal{R}%
(\lambda,\overrightarrow{a},f).
\end{align*}
By symmetry $\mathcal{R}(\lambda,\overrightarrow{a},f)=\mathcal{R}%
(\lambda,f,\overrightarrow{a})$, whence
\begin{align*}
\left\vert (T(\lambda)f,f)\right\vert  &  \leq\left\Vert (\Sigma
^{-1}-\mathcal{R}(\lambda,\overrightarrow{a},\overrightarrow{a}))^{-1}%
\right\Vert \left\Vert \mathcal{R}(\lambda,\overrightarrow{a},f)\right\Vert
^{2}\\
&  \leq\left\Vert (\Sigma^{-1}-\mathcal{R}(\lambda,\overrightarrow
{a},\overrightarrow{a}))^{-1}\right\Vert \left\Vert (L-\lambda\mathrm{I}%
)^{-1}f\right\Vert ^{2}\\
&  \leq\left\Vert (\Sigma^{-1}-\mathcal{R}(\lambda,\overrightarrow
{a},\overrightarrow{a}))^{-1}\right\Vert \left\Vert (L-\lambda\mathrm{I}%
)^{-1}\right\Vert ^{2}\left\Vert f\right\Vert ^{2}%
\end{align*}
as desired. The proof of the theorem is finished.
\end{proof}

\section{Sparse potentials}

We assume that the ultrametric measure space $(X,d,m)$ is \emph{countably
infinite and homogeneous. }Our analysis of finite rank potentials
$V=-\sum_{i=1}^{N}\sigma_{i}\delta_{a_{i}}$ indicates that in the case of
increasing distances between locations $\{a_{i}\}$ of the \emph{bumps}
$V_{i}=-\sigma_{i}\delta_{a_{i}}$ their contributions to the spectrum of
$H=L+V$ is close to the union of the contributions of the individual bumps
$V_{i}$ (each bump contributes one eigenvalue in each gap $(\lambda
_{m+1},\lambda_{m})$ of the spectrum of the operator $L$).

The development of this idea leads to consideration of the class of
\emph{sparse potentials }$V=-\sum_{i=1}^{\infty}\sigma_{i}\delta_{a_{i}}$
where distances between locations $\{a_{i}:i=1,2,...\}$ form a fast increasing
sequence. In the classical theory this idea goes back to D. B. Pearson
\cite{Pearson}, see also S. Molchanov \cite{Molchanov2} and A. Kiselev, J.
Last, S. and B. Simon \cite{KiselevLastSimon}.

Throughout this section we will assume that the sequence $\min_{i,j:\geq
n,\text{ }i\neq j}\mathrm{d}(a_{i},a_{j})$ tend to infinity with certain rate
which will be specified later\footnote{We choose the ultrametric
\textrm{d}$(x,y)$ such that it coinsides with the measure $m(B)$ of the
minimal ball $B$ which contains both $x$ and $y$, see e.g. (\ref{A(B)}).}. We
will also assume that $\alpha<\sigma_{i}<\beta$ for all $i$ and for some
$\alpha,\beta>0$. For $\lambda\notin Spec(L)$ we define the following infinite
vectors and matrices:

\begin{itemize}
\item $\mathcal{R}(\lambda,x,\overrightarrow{a}):=(\mathcal{R}(\lambda
,x,a_{i}):i=1,2,...).$

\item $\mathcal{R}(\lambda,\overrightarrow{a},\overrightarrow{a}%
):=(\mathcal{R}(\lambda,a_{i},a_{j}):i,j=1,2,...).$

\item $\Sigma:=\mathrm{diag}(\sigma_{i}:i=1,2,...),$ $\Sigma^{-1}%
:=\mathrm{diag}(1/\sigma_{i}:i=1,2,...).$
\end{itemize}

\begin{theorem}
\label{l^2-thm}The following properties hold true:

\begin{description}
\item[(i)] $\mathcal{R}(\lambda,x,\overrightarrow{a})\in l^{2}.$

\item[(ii)] $\mathcal{R}(\lambda,\overrightarrow{a},\overrightarrow{a})$,
$\Sigma$ and $\Sigma^{-1}$ act in $l^{2}$ as bounded symmetric operators.

\item[(iii)] If the operator $\mathfrak{B}(\lambda)=\Sigma^{-1}-\mathcal{R}%
(\lambda,\overrightarrow{a},\overrightarrow{a})$ has a bounded inverse, then
\begin{equation}
\mathcal{R}_{V}(\lambda,x,y)=\mathcal{R}(\lambda,x,y)+\mathcal{R}%
(\lambda,x,\overrightarrow{a})\mathfrak{B}(\lambda)^{-1}\mathcal{R}%
(\lambda,\overrightarrow{a},y). \label{R_V infty}%
\end{equation}

\end{description}
\end{theorem}

\begin{proof}
Let $\xi=(\xi_{i})\in l^{2}$ has finite number non-zero coordinates. Define
function $f=\sum\xi_{i}\delta_{a_{i}}$. Evidently $f\in L_{2}=L_{2}(X,m)$ and
$\left\Vert f\right\Vert =\left\Vert \xi\right\Vert .$Let $R_{\lambda
}=(L-\lambda\mathrm{I})^{-1}$, $\lambda\notin Spec(L)$, be the resolvent. Then%
\[
\mathcal{R}(\lambda,x,\overrightarrow{a})\xi=\int\mathcal{R}(\lambda
,x,y)f(y)dm(y)=R_{\lambda}f(x)
\]
whence%
\[
\left\vert \mathcal{R}(\lambda,x,\overrightarrow{a})\xi\right\vert
\leq\left\Vert R_{\lambda}\right\Vert \left\Vert f\right\Vert =\left\Vert
R_{\lambda}\right\Vert \left\Vert \xi\right\Vert
\]
which clearly proves $(i).$To prove $(ii)$ we write%
\begin{align*}
\xi\mathcal{R}(\lambda,\overrightarrow{a},\overrightarrow{a})\xi &  =\int\int
f(x)\mathcal{R}(\lambda,x,y)f(y)dm(y)dm(x)\\
&  =(f,R_{\lambda}f)\leq\left\Vert R_{\lambda}\right\Vert \left\Vert
f\right\Vert ^{2}=\left\Vert R_{\lambda}\right\Vert \left\Vert \xi\right\Vert
^{2}%
\end{align*}
which clearly proves boundedness of the symmetric operator $\mathcal{R}%
(\lambda,\overrightarrow{a},\overrightarrow{a}):l^{2}\rightarrow l^{2}$. Since
$\{\sigma_{i}\}\in(\alpha,\beta)$ for all $i$ and some $\alpha,\beta>0$,
boundedness of the operators $\Sigma$ and $\Sigma^{-1}$ follows.

$(iii)$ Assume that $\lambda$ is such that the self-adjoint operator
$\mathfrak{B}(\lambda)$ has a bounded inverse, then equation (\ref{R_V infty})
follows from its finite dimensional version (\ref{R_V}) by passage to limit.
\end{proof}

\begin{theorem}
\label{SpecL in SpecH}$Spec(L)\subset Spec_{ess}(H).$
\end{theorem}

\begin{proof}
Let $V^{\prime}$ be the sum of all but finite number of bumps $V_{i}$ and
$H^{\prime}=L+V^{\prime}$. By Weyl's essential spectrum theorem $Spec_{ess}%
(H)=Spec_{ess}(H^{\prime})$. It follows that without loss of generality we may
assume that the sequence of distances $\ \Delta_{n}=$ $\min_{i,j:\geq n,\text{
}i\neq j}\mathrm{d}(a_{i},a_{j})$ strictly increases to $\infty$. Having this
in mind we can choose for any given $\tau$ from the range of the distance
function an infinite sequence $\{B_{n}\}$ of disjoint balls of diameter $\tau$
such that $B_{n}\cap\{a_{i}\}=\varnothing$ for all $n$. Thanks to our choice
we obtain
\[
Hf_{T}=Lf_{T}=\lambda(T^{\prime})f_{T}%
\]
for any ball $T\subset B_{n}$ and for all $n$. In particular, each
$\lambda=\lambda(T),$ such that $T\subseteq B_{n}$ for some $n,$ is an
eigenvalue of the operator $H$ having infinite multiplicity, whence it belongs
to $Spec_{ess}(H).$
\end{proof}

\begin{theorem}
\label{sigma-limit point}Let $\sigma_{\ast}$ be a limit point of the sequence
$\{\sigma_{i}\}$. Fix $m\in\mathbb{N}$ and let $\lambda_{\ast m}\in
(\lambda_{m+1},\lambda_{m})$ be the unique solution of the equation%
\begin{equation}
\frac{1}{\sigma_{\ast}}=\mathcal{R}(\lambda,a,a).\text{ }\footnote{Recall that
the function $\lambda\rightarrow$ $\mathcal{R}(\lambda,a,a)$ does not depend
on $a$.} \label{sigma-star-equation}%
\end{equation}
Then $\lambda_{\ast m}$ belongs to the set $Spec_{ess}(H)$.
\end{theorem}

Before we embark on the proof of Theorem \ref{sigma-limit point} let us state
the Weyl's characterization of the essential spectrum $Spec_{ess}(A)$ of a
self-adjoint operator $A$, see \cite{Weyl} and \cite[Ch. IX, Sect.
2(133)]{Riesz}.

\begin{lemma}
\label{Weyl's lemma} A real number $\lambda$ belongs to the set $Spec_{ess}%
(A)$ if and only if there exists a normed sequence $\{x_{i}\}\subset
\mathrm{dom}(A)$ such that $x_{i}\rightarrow0$ weakly and $Ax_{i}-\lambda
x_{i}\rightarrow0$ strongly.
\end{lemma}

\textbf{Proof of Theorem} \ref{sigma-limit point}. To show that $\lambda_{\ast
m}\in$ $Spec_{ess}(H)$ we construct a $\lambda_{\ast m}$-sequence $\{f_{im}\}$
via Lemma \ref{Weyl's lemma}. Let $\lambda_{im}\in(\lambda_{m+1},\lambda_{m})$
be the unique solution of the equation $1/\sigma_{i}=\mathcal{R}(\lambda
,a_{i},a_{i})$. Let $\psi_{im}(x)=\mathcal{R}(\lambda_{im},x,a_{i})/\left\Vert
\mathcal{R}(\lambda_{im},\cdot,a_{i})\right\Vert _{2}$ be the normed solution
of the equation $H_{i}\psi=\lambda_{im}\psi$ where $H_{i}:=L-\sigma_{i}%
\delta_{a_{i}}$ is a one-bump perturbation of $L$. Clearly $\lambda
_{im}\rightarrow\lambda_{\ast m}$.

Passing if necessary to a subsequence of $\{\sigma_{i}\}$ we can assume that
\textrm{d}$(a_{i},0)\rightarrow\infty$ monotonically. Let us put $f_{im}%
:=\psi_{im}\cdot1_{B_{i}}$ where $B_{i}$ is the maximal ball centred at
$a_{i}$ which does not contains $a_{i-1}$ and $a_{i+1}$. Thanks to our choice
$f_{im}\rightarrow0$ weakly and%
\[
\left\Vert f_{im}\right\Vert _{2}^{2}=\int_{B_{i}}\left\vert \psi
_{im}\right\vert ^{2}dm\rightarrow1.
\]
Thus what is left is to show that $Hf_{im}-\lambda_{\ast m}f_{im}\rightarrow0$
strongly. We have%
\begin{align*}
\left\Vert Hf_{im}-\lambda_{\ast m}f_{im}\right\Vert _{2}  &  \leq\left\Vert
Hf_{im}-\lambda_{im}f_{im}\right\Vert _{2}+\left\Vert f_{im}\right\Vert
_{2}|\lambda_{im}-\lambda_{\ast m}|\\
&  \leq\left\Vert Hf_{im}-\lambda_{im}f_{im}\right\Vert _{2}+|\lambda
_{im}-\lambda_{\ast m}|\\
&  =\left\Vert Hf_{im}-\lambda_{im}f_{im}\right\Vert _{2}+o(1),
\end{align*}%
\begin{align*}
\left\Vert Hf_{im}-\lambda_{im}f_{im}\right\Vert _{2}  &  \leq\left\Vert
H\psi_{im}-\lambda_{im}\psi_{im}\right\Vert _{2}+\left\Vert (H-\lambda
_{im}\mathrm{I})(f_{im}-\psi_{im})\right\Vert _{2}\\
&  \leq\left\Vert H\psi_{im}-\lambda_{im}\psi_{im}\right\Vert _{2}+\left\Vert
(H-\lambda_{im}\mathrm{I})||\text{ }||(f_{im}-\psi_{im})\right\Vert _{2}\\
&  =\left\Vert H\psi_{im}-\lambda_{im}\psi_{im}\right\Vert _{2}+o(1),
\end{align*}%
\[
\left\Vert H\psi_{im}-\lambda_{im}\psi_{im}\right\Vert _{2}\leq\left\Vert
H_{i}\psi_{im}-\lambda_{im}\psi_{im}\right\Vert _{2}+\left\Vert \sum_{j\neq
i}\sigma_{j}\delta_{a_{j}}\psi_{im}\right\Vert _{2}%
\]
and%
\[
\left\Vert \sum_{j\neq i}\sigma_{j}\delta_{a_{j}}\psi_{im}\right\Vert
_{2}=\sqrt{\sum_{j\neq i}\sigma_{j}^{2}|\psi_{im}(a_{j})|^{2}}\leq\sup
\{\sigma_{j}^{2}\}\sqrt{\int_{X\setminus B_{i}}|\psi_{im}|^{2}dm}.
\]
The right-hand side of this inequality tends to zero as $i\rightarrow\infty$
and we finally conclude that $\{f_{im}\}$ is the desired $\lambda_{\ast m}%
$-sequence in the sense of Lemma \ref{Weyl's lemma}. The proof is finished.

Let us introduce the following notation

\begin{itemize}
\item $\Sigma_{\ast}$ is the set of limit points of the sequence $\{\sigma
_{i}\}$

\item $1/\Sigma_{\ast}:=\{1/\sigma_{\ast}:\sigma_{\ast}\in\Sigma_{\ast}\}$

\item $\mathcal{R}^{-1}(1/\Sigma_{\ast}):=\{\lambda:\mathcal{R}(\lambda
,a,a)\in1/\Sigma_{\ast}\}$
\end{itemize}

\begin{theorem}
\label{Spec_ess}Assume that the following condition holds
\begin{equation}
\lim_{N\rightarrow\infty}\sup_{i\geq N}\sum_{j\geq N:\text{ }j\neq i}\frac
{1}{\mathrm{d}(a_{i},a_{j})}=0, \label{a_ij-condition}%
\end{equation}
then
\begin{equation}
Spec_{ess}(H)=Spec(L)\cup\mathcal{R}^{-1}(1/\Sigma_{\ast})\text{. }
\label{spectral equation}%
\end{equation}

\end{theorem}

\begin{proof}
That $Spec(L)$ and $\mathcal{R}^{-1}(1/\Sigma_{\ast})$ are subsets of
$Spec_{ess}(H)$ follows from Theorem \ref{SpecL in SpecH} and Theorem
\ref{sigma-limit point}. We are left to prove that
\[
Spec_{ess}(H)\subset Spec(L)\cup\mathcal{R}^{-1}(1/\Sigma_{\ast}).
\]
Let us fix $m\in\mathbb{N}$ and choose a closed interval $\mathcal{I}$ from
the spectral gap $(\lambda_{m+1},\lambda_{m})$. We claim that
\[
\mathcal{I}\cap Spec_{ess}(H)=\varnothing.
\]
Indeed, since $\mathcal{R}(\lambda):=\mathcal{R}(\lambda,a,a)$ is strictly
increasing and continuous in the interval $(\lambda_{m+1},\lambda_{m})$,
closed sets $\mathcal{R(I)}$ and $1/\Sigma_{\ast}$ do not intersect. Hence
there exists only a finite number of $\sigma_{i}$ such that $1/\sigma_{i}%
\in\mathcal{R(I)}$. Let us choose $N$ big enough so that the sets
$\{1/\sigma_{i}:i>N\}$ and $\mathcal{R(I)}$ do not intersect. Let us write
$H=H^{\prime}+V^{\prime}$ where $V^{\prime}$ is a finite number of bumps
$-\sigma_{i}\delta_{a_{i}}$, $i\leq N$. By Weyl's essential spectrum theorem
\[
Spec_{ess}(H)=Spec_{ess}(H^{\prime}).
\]
Notice however that the sets $Spec_{d}(H)$ and $Spec_{d}(H^{\prime})$,
discrete parts of $Spec(H)$ and $Spec(H^{\prime})$, may well be quite
different. Observe that for the operators $H$ and $H^{\prime}$ the sets of
limit points, the function $\mathcal{R}$, the set of gaps etc are the same.
Thus in all our further considerations we may \emph{assume that}
$\{1/\sigma_{i}\}\cap\mathcal{R(I)=\varnothing}$.

Making this assumption consider now the operator $\mathfrak{B}(\lambda
)=\Sigma^{-1}-\mathcal{R}(\lambda,\overrightarrow{a},\overrightarrow{a})$,
$\lambda\in\mathcal{I}$. According to identity (\ref{R_V infty}), if
$\mathfrak{B}(\lambda)$ has a bounded inverse then $\lambda\notin Spec(H)$.
Let us write%
\[
\mathfrak{B}(\lambda)=\Sigma^{-1}-\mathcal{R}(\lambda,\overrightarrow
{a},\overrightarrow{a}):=\left[  \Sigma^{-1}-\mathcal{R}(\lambda
)\mathrm{I}\right]  -\widetilde{\mathcal{R}}(\lambda).
\]
Since we assume that the closed bounded sets $\overline{\{1/\sigma_{i}\}}$ and
$\mathcal{R(I)}$ do not intersect, the operator $\mathcal{A(\lambda)}%
:=\Sigma^{-1}-\mathcal{R}(\lambda)\mathrm{I}$ has a bounded inverse
$\mathcal{A(\lambda)}^{-1}$ for all $\lambda\in$ $\mathcal{I}$. Clearly the
norm$\left\Vert \mathcal{A(\lambda)}^{-1}\right\Vert $ can be estimated by the
reciprocal of the distance between sets $\overline{\{1/\sigma_{i}\}}$ and
$\mathcal{R(I)}$, denote it by $C_{1}$. Thus writing for $\lambda
\in\mathcal{I}$ the identity
\begin{equation}
\mathfrak{B}(\lambda)=\mathcal{A(\lambda)}(\mathrm{I}-\mathcal{A(\lambda
)}^{-1}\widetilde{\mathcal{R}}(\lambda)) \label{B-operator}%
\end{equation}
we get%
\begin{equation}
\left\Vert \mathcal{A(\lambda)}^{-1}\widetilde{\mathcal{R}}(\lambda
))\right\Vert \leq\mathcal{C}_{1}\left\Vert \widetilde{\mathcal{R}}%
(\lambda))\right\Vert . \label{A-R inequality}%
\end{equation}
Writing again $H$ as $H^{\prime}+V^{\prime}$ where $V^{\prime}$consists of a
finite number, say $N$, of bumps and applying inequality (\ref{R-lambda-ineq})
for the operator $H^{\prime}:$
\[
\left\vert \mathcal{R}(\lambda,a_{i},a_{j})\right\vert <\frac{1}%
{\mathrm{d}(a_{i},a_{j})}\frac{1}{\lambda-\lambda(a_{i}\curlywedge a_{j}%
)},\text{ }i\neq j,\text{ }i,j\geq N,
\]
we will get, thanks to our assumption (\ref{a_ij-condition}), the following
inequality
\begin{equation}
\left\Vert \widetilde{\mathcal{R}}(\lambda))\right\Vert \leq C_{2}\sup_{i\geq
N}\sum_{j:\text{ }j\neq i,j\geq N}\frac{1}{\mathrm{d}(a_{i},a_{j})}<\frac
{1}{2C_{1}} \label{A-R inequality'}%
\end{equation}
for some constant $C_{2}>0$ which depends only on $\mathcal{I}$, and for $N$
chosen big enough. Clearly inequalities (\ref{A-R inequality}) and
(\ref{A-R inequality'}) imply the fact that the operator $\mathrm{I}%
-\mathcal{A(\lambda)}^{-1}\widetilde{\mathcal{R}}(\lambda)$ has bounded
inverse for all $\lambda\in\mathcal{I}$,%
\[
\left(  \mathrm{I}-\mathcal{A(\lambda)}^{-1}\widetilde{\mathcal{R}}%
(\lambda)\right)  ^{-1}=\sum_{k\geq0}\left(  \mathcal{A(\lambda)}%
^{-1}\widetilde{\mathcal{R}}(\lambda)\right)  ^{k}.
\]
This fact, in turn, implies that the operator $\mathfrak{B}(\lambda)$ given by
equation (\ref{B-operator}) has bounded inverse for all $\lambda\in
\mathcal{I}$ therefore $\mathcal{I}\cap Spec(H^{\prime})=\varnothing$. In
particular, since $Spec_{ess}(H^{\prime})=Spec_{ess}(H)$ by Weyl's essential
spectrum theorem, we finally get
\[
\mathcal{I}\cap Spec_{ess}(H)=\varnothing
\]
as desired. The proof is finished.
\end{proof}

\begin{remark}
Theorem \ref{Spec_ess} does not contain information about sets $Spec_{ac}(H)$
and $Spec_{sc}(H)$, the absolutely continuous and singular continuous parts of
$Spec(H)$. In the next section we will show that under more restrictive
assumption $Spec_{ac}(H)$ and $Spec_{sc}(H)$ are indeed empty sets, that is,
$Spec(H)$ is pure point. Moreover, the eigenfunctions of $H$ decay
exponentially in certain metric at infinity. This is the so called
\emph{localization property}.
\end{remark}

\section{Localization}

As in the previous section the ultrametric measure space $(X,d,m)$ is
\emph{countably infinite and homogeneous. }We consider the operator $H=L+V$
where $L$, the deterministic part of $H$, is a hierarchical Laplacian and%
\[
V=-\sum_{a\in I}\sigma(a,\omega)\delta_{a},\text{ }\omega\in(\Omega
,\mathcal{F},P),
\]
is a random potential defined by a family of locations $I=\{a_{i}\}$ and a
family $\sigma(a_{i},\omega)$ of i.i.d. random variables. Henceforth, we
assume that the probability distribution of $\sigma(a_{i},\omega)$ is
absolutely continuous with respect to the Lebesgue measure and has a bounded
density supported by a finite interval $[\alpha,\beta]$.

In the case when $X$ is the Dyson lattice and $L=$\textrm{D}$^{\alpha}$, the
Dyson Laplacian (see Example \ref{Example "One-point pert."}), the perturbed
operator%
\[
H=\mathrm{D}^{\alpha}-\sum_{a\in X}\sigma(a,\omega)\delta_{a}%
\]
has a pure point spectrum for $P-$a.s. $\omega.$ This statement (\emph{the
localization theorem}) appeared first in the paper of Molchanov
\cite{Molchanov} ($\sigma(a,\omega)$ is the Cauchy random variable) and later
in a more general form in the papers of Kritchevski \cite{Kvitchevski2} and
\cite{Kvitchevski1}. The proof of this statement is based on the
self-similarity property of the operator $H$.

The localization theorem \ref{AisenMol} below concerns the case where the
family of locations $I$ does not coincide with the whole space $X$, whence the
operator $H$ is not self-similar. The technique developed in \cite{Molchanov},
\cite{Kvitchevski2} and \cite{Kvitchevski1} does not apply here to prove
Theorem \ref{AisenMol}.

Our approach is based on the different technique: the abstract form of the
Aizenman-Molchanov criterion for pure point spectrum, the Krein-type identity
from the previous section, technique of fractional moments, decoupling lemma
of Molchanov and Borel-Cantelli type arguments, see papers
\cite{AisenmanMolchanov}, \cite{Molchanov1}.

\paragraph{The Aizenman-Molchanov Criterion}

Let $H=H_{0}+V$ be a self-adjoint operator in $l^{2}(\Gamma)$ ($\Gamma$ is a
countable set of sites) with $H_{0}$ a bounded operator and $V=-\sum
_{a\in\Gamma}\sigma(a,\omega)\delta_{a}$. Assume that the collection of random
variables $\{\sigma(a,\omega):a\in\Gamma\}$ has the property that for each
site $a$ the conditional probability distribution of $\sigma(a,\omega)$
(conditioned on the values of the potential at all other sites) is absolutely
continuous with respect to the Lebesgue measure (in particular, this
assumption holds if $\{\sigma(a,\omega):a\in\Gamma\}$ are mutually independent
random variables having absolutely continuous w.r.t. the Lebesgue measure $l$
probability distributions).

Let $H=\int\lambda dE_{\lambda}$ be the spectral resolution of symmetric
operator $H$. Let $G(\lambda,x,y)$ be the integral kernel of the operator
$(H-\lambda\mathrm{I})^{-1}$. Then for any fixed $x,\tau$ and $\epsilon\neq0$,%
\begin{equation}
\sum_{y\in\Gamma}|G(\tau+i\epsilon,x,y)|^{2}=\left\Vert (H-(\tau
+i\epsilon)\mathrm{I})^{-1}\delta_{x}\right\Vert ^{2}=%
{\displaystyle\int}
\frac{d(E_{\lambda}\delta_{x},\delta_{x})}{(\lambda-\tau)^{2}+\epsilon^{2}}
\label{G-tau-upsilon eq.}%
\end{equation}
As the left-hand side of equation (\ref{G-tau-upsilon eq.}) (as a function of
$\epsilon$) decreases on the interval $]0,+\infty\lbrack$, the limit (finite
or infinite) in equation (\ref{G-tau-upsilon eq.}) exists and equals
\[
\lim_{\epsilon\downarrow0}\sum_{y\in\Gamma}|G(\tau+i\epsilon,x,y)|^{2}=%
{\displaystyle\int}
\frac{d(E_{\lambda}\delta_{x},\delta_{x})}{(\lambda-\tau)^{2}}.
\]

\begin{theorem}
\label{AizMol} If for any $x\in\Gamma$, and Lebesgue a.a. $\tau\in\lbrack
a,b]$:%
\begin{equation}
\lim_{\epsilon\downarrow0}%
{\displaystyle\sum\limits_{y\in\Gamma}}
\left\vert G(\tau+i\epsilon,x,y)\right\vert ^{2}<\infty,
\label{cyclic vector condition}%
\end{equation}
for a.e. realizations of $\{\sigma(x,\cdot)\}$, then almost surely the
operator $H$ has only pure point spectrum in the interval $[a,b]$.
Furthermore, if under condition (\ref{cyclic vector condition}), the integral
kernel
\[
G(\tau+i0,x,y):=\lim_{\epsilon\downarrow0}G(\tau+i\epsilon,x,y)\text{ }%
\]
(which exists a.e. $\tau$) decays exponentially at infinity (in some metric
$\rho(x,y)$ on $\Gamma$), then do the eigenfunctions $\varphi_{\tau}(y)$, for
$\tau\in\lbrack a,b]$ \footnote{An even more versatile version can be found in
\cite[Theorems 3.1 and 3.3 in Sec. 3]{AisenmanMolchanov}.}.
\end{theorem}

\begin{proof}
The first part of the statement follows from Simon-Wolff theorem \ref{SiWo}.
For completeness of exposition we comment on the proof. To prove the second
part one needs an ad hoc argument and we refer to the cited above paper
\cite[Theorems 3.1 and 3.3 in Sec. 3]{AisenmanMolchanov}).

Note that in the case of the Dyson-Vladimirov Laplacian $\mathrm{D}^{\alpha}$
and $H=\mathrm{D}^{\alpha}-\sum\sigma_{i}(a_{i},\omega)\delta_{a}$ one can use
the metric $\rho(x,y)=\ln(1+\mathrm{d}(x,y))$ where $\mathrm{d}(x,y)$ is the
ultrametric generated by $p$-adic intervals as in example
\ref{Example "One-point pert."}. In this case the exponential decay of
eigenfunctions in $\rho-$metric follows directly from two facts: (1) each
eigenfunction $\varphi_{\tau}(y)$ of $H$ can be represented as a linear
combination of functions $\mathcal{R}(\tau,a_{i},y)$, where $\mathcal{R}%
(\lambda,x,y)$ is the resolvent kernel of $\mathrm{D}^{\alpha}$, see Theorem
\ref{E-f for discr.spec.}, and (2) $\mathcal{R}(\lambda,x,y)$ has an
exponential decay because the heat kernel $p(t,x,y)$ does, see equation
(\ref{Dyson_heat_kernel}).

By the spectral theory, one can represent $l^{2}(\Gamma)$ as the direct sum of
three $H$-invariant subspaces:
\[
l^{2}(\Gamma)=\mathcal{H}_{ac}\oplus\mathcal{H}_{sc}\oplus\mathcal{H}_{pp},
\]
where $\mathcal{H}_{ac}$ (resp. $\mathcal{H}_{sc},$ $\mathcal{H}_{pp}$) is the
set of all functions $f\in$ $l^{2}(\Gamma)$ such that the spectral measure
\[
\sigma^{f}(A)=\int1_{A}(\lambda)d(E_{\lambda}f,f)
\]
is absolutely continuous (resp. singular continuous, pure point) with respect
to the Lebesgue measure. By Theorem \ref{SiWo}, condition
(\ref{cyclic vector condition}) implies that for any $x\in\Gamma$ the
probability measure
\[
\sigma^{x}(A)=\int1_{A}(\lambda)d(E_{\lambda}\delta_{x},\delta_{x})
\]
is pure point, that is, $\sigma^{x}(A)=\sigma^{x}(A\cap S_{x})$ for any open
set $A$ and some at most countable set $S_{x}$. Set $S:=\cup_{x\in\Gamma}%
S_{x}$, then for any $f\in l^{2}(\Gamma)$ and measurable set $A$,
\begin{align*}
\sigma^{f}(A)  &  =\int1_{A}(\lambda)d(E_{\lambda}f,f)=\left\Vert
1_{A}(H)f\right\Vert ^{2}\\
&  =\sum_{x\in\Gamma}\left\vert f(x)\right\vert ^{2}\left\vert (1_{A}%
(H)f,\delta_{x})\right\vert ^{2}%
\end{align*}
and, if $A$ lies in the complement of $S,$%
\begin{align*}
\left\vert (1_{A}(H)f,\delta_{x})\right\vert ^{2}  &  \leq\left\vert
(1_{A}(H)f,f)\right\vert \left\vert (1_{A}(H)\delta_{x},\delta_{x})\right\vert
\\
&  =\left\Vert 1_{A}(H)f\right\Vert ^{2}\sigma^{x}(A)=0.
\end{align*}
Thus for any $f\in l^{2}(\Gamma)$ the spectral measure $\sigma^{f}$ is pure
point, that is, $f\in\mathcal{H}_{pp}$. That means that the operator $H$ has a
pure point spectrum.
\end{proof}

\begin{remark}
The function $z\rightarrow G(z,x,y)$, analytic in the domain $%
\mathbb{C}
_{+}$, is represented by the Borel-Stieltjes transform of a signed measure of
finite variation%
\[
G(z,x,y)=\int\frac{d(E_{\lambda}\delta_{x},\delta_{y})}{\lambda-z}.
\]
It follows that the limit $G(\tau+i0,x,y)$ exists and takes finite values for
Lebesgue a.e. $\tau$, see e.g. \cite[Theorem 1.4]{Simon}. Moreover, the limit
$G(\tau+i0,x,y)$ exists even in a more restrictive sense, as the
non-tangential limit, see \cite[Ch. III, Sec. 2.2, 3.1 and 3.2]{Privalov}. We
will apply this fact in the proof of Theorem \ref{AisenMol} below.
\end{remark}

\paragraph{The localization theorem}

Coming back to our setting, let $H=L+V$ where $L$ is a hierarchical Laplacian
and $V$ a random potential of the form $V=-\sum_{i}\sigma_{i}(\omega
)\delta_{a_{i}}$. Here $\sigma_{i}(\omega):=\sigma(a_{i},\omega)$ are i.i.d.
random variables corresponding to the set of locations $I=\{a_{i}\}$.

Let $\mathrm{d}(x,y)$ be the ultrametric which is chosen such that it
coincides with the measure $m(B)$ of the minimal ball $B$ containing both $x$
and $y$.

Let $\mathcal{R}(\lambda,x,y)$ be the integral kernel of the operator
$(L-\lambda\mathrm{I})^{-1}$, i.e. the solution of the equation $Lu-\lambda
u=\delta_{y}$. The function $\lambda\rightarrow\mathcal{R}(\lambda,x,x)$ does
not depend on $x$, we denote its value $\mathcal{R}(\lambda)$. This is
strictly increasing continuous in each spectral gap function, we denote by
$\mathcal{R}^{-1}(\nu)$ its inverse function.

\begin{theorem}
\label{AisenMol}The operator $H$ has a pure point spectrum for $P-$a.s.
$\omega$ provided for some (whence for all) $y\in X$ the sequence
$\mathrm{d}(a_{i},y)$ eventually increases, and for some small $r$ (say,
$0<r<1/3$):
\begin{equation}
\lim_{M\rightarrow\infty}\sup_{i\geq M}\sum_{j\geq M:\text{ }j\neq i}\frac
{1}{\mathrm{d}(a_{i},a_{j})^{r}}=0.\text{ }\footnote{Clearly this condition
implies condition (\ref{a_ij-condition}).} \label{a-ij-s condition}%
\end{equation}

\end{theorem}

\begin{proof}
The set of limit points of the sequence $\{\sigma_{i}(\omega)\}$ coincides
(for $P-$a.s. $\omega$) with the whole interval $[\alpha,\beta]$. Hence, by
Theorem \ref{Spec_ess}, the closed set $Spec_{ess}(H)$ consists (for $P$-a.a.
$\omega$) of two parts: (1) the set $Spec(L)$ and (2) the collection of
countably many disjoint closed intervals $\mathcal{I}_{k}=\mathcal{R}%
^{-1}([1/\beta,1/\alpha])\cap$ $]\lambda_{k+1},\lambda_{k},[$ and the interval
$\mathcal{I}_{-}=\mathcal{R}^{-1}([1/\beta,1/\alpha])\cap$ $]-\infty,0[$,
i.e.
\[
Spec_{ess}(H)=Spec(L)\cup\mathcal{I}_{-}\cup\mathcal{I}_{1}\cup\mathcal{I}%
_{2}...\mathcal{\ }.
\]
Let $\mathcal{R}_{V}(\lambda,x,y)$ be the integral kernel of the operator
$(H-\lambda\mathrm{I})^{-1}$, i.e. solution of the equation $Hu-\lambda
u=\delta_{y}$. Due to the Aizenman-Molchanov criterion, the operator $H$ has
only pure point spectrum (for $P-$a.s. $\omega$) provided \emph{ for each
}$y\in X$, \emph{ for each interval }$\mathcal{I}_{k},$ and for \emph{
Lebesgue a.e.} $\tau\in\mathcal{I}_{k}:$
\begin{equation}
\lim_{\epsilon\rightarrow+0}\sum_{x}|\mathcal{R}_{V}(\tau+i\epsilon
,x,y)|^{2}<\infty\text{, } \label{Simon-Wolf-Theorem}%
\end{equation}
for a.e. realization of $\{\sigma(y,\omega)\}.$We split the proof of equation
(\ref{Simon-Wolf-Theorem}) in seven steps.

\underline{Step I. } When $V$ has a finite rank, Theorem
\ref{pure point spectrum}$(i)$ and Theorem \ref{perturbed resolvent} imply
that for each fixed $\omega$ the function $\mathcal{R}_{V}(\tau
+i0,x,y)=(H-\tau\mathrm{I})^{-1}\delta_{y}(x)$ belongs to $L^{2}(X,m)$ for
each $y\in X$ and for all but finitely many $\tau\in\mathcal{I}_{k}$ (which
are eigenvalues of $H$).

In general, when the rank of $V$ is infinite, we split $V$ in two parts
$V^{\prime}=-\sigma_{1}\delta_{a_{1}}$ and $V^{\prime\prime}=-\sum_{i>1}%
\sigma_{i}\delta_{a_{i}}$. Writing the set of locations as $\{a\}=\{a_{1}%
\}\cup\{a_{i}:i>1\}$ we get similarly to equation (\ref{R_V infty}): for
$\lambda$ in the domain $%
\mathbb{C}
_{+}$,%
\[
\mathcal{R}_{V}(\lambda,x,y)=\mathcal{R}_{V^{\prime\prime}}(\lambda
,x,y)+\mathcal{R}_{V^{\prime\prime}}(\lambda,x,a_{1})\mathfrak{B}%
(\lambda)^{-1}\mathcal{R}_{V^{\prime\prime}}(\lambda,a_{1},y),
\]
where $\mathfrak{B}(\lambda)=1/\sigma_{1}-\mathcal{R}_{V^{\prime\prime}%
}(\lambda,a_{1},a_{1})$ is a non-constant analytic in the domain $%
\mathbb{C}
_{+}$ function. It follows that
\begin{align*}
\left\Vert \mathcal{R}_{V}(\lambda,\cdot,y)\right\Vert _{2}  &  \leq\left\Vert
\mathcal{R}_{V^{\prime\prime}}(\lambda,\cdot,y)\right\Vert _{2}\\
&  +|\mathfrak{B}(\lambda)|^{-1}\left\Vert \mathcal{R}_{V^{\prime\prime}%
}(\lambda,\cdot,a_{1})\right\Vert _{2}|\mathcal{R}_{V^{\prime\prime}}%
(\lambda,a_{1},y)|.
\end{align*}
Hence the function $\mathcal{R}_{V}(\lambda,x,y)$ satisfies condition
(\ref{Simon-Wolf-Theorem}), i.e. $\left\Vert \mathcal{R}_{V}(\tau
+i0,\cdot,y)\right\Vert _{2}$ is finite for all $y$ and a.e. $\tau$ provided
the function $\mathcal{R}_{V^{\prime\prime}}(\lambda,x,y)$ satisfies condition
(\ref{Simon-Wolf-Theorem}), i.e. $\left\Vert \mathcal{R}_{V^{\prime\prime}%
}(\tau+i0,\cdot,a)\right\Vert _{2}$ is finite for all $a$ and a.e. $\tau$, and
also one more restriction on $\tau,$ it does not belong to the exceptional
set
\[
\Upsilon:=\{s:\mathfrak{B}(s+i0)=0\}.
\]
The function $\mathfrak{B}(\lambda)$, analytic in the domain $%
\mathbb{C}
_{+}$, admits non-tangential boundary values $\mathfrak{B}(s+i0)$ for a.e.
$s$. By the Lusin-Privalov uniqueness theorem on boundary-values of analytic
functions \cite[Ch. IV, Sec. 2.5]{Privalov}, see also \cite[Theorem
1.5]{Simon}, the \emph{Lebesgue measure }of\emph{ }the exceptional set
$\Upsilon$ equals to zero. Thus, we come to the conclusion that condition
(\ref{Simon-Wolf-Theorem}) for the potential $V$ can be reduced to the case of
truncated potential $V^{\prime\prime}$.

Repeating this argument finitely many times we come to the final conclusion:
in order to prove that (\ref{Simon-Wolf-Theorem}) holds for $V$ we can
consider, if necessary, any \emph{finitely} \emph{truncated} potential
$V^{\prime\prime}$ (the potential corresponding to the finitely truncated
system of locations $\{a_{i}:i>k\}$) and to prove that
(\ref{Simon-Wolf-Theorem}) holds for $V^{\prime\prime}$ instead of $V$.

\underline{Step II. } Writing for $\lambda\in%
\mathbb{C}
_{+}$ equation $Hu-\lambda u=\delta_{y}$ in the form $Lu-\lambda u=\delta
_{y}-Vu$ we obtain%
\begin{equation}
\mathcal{R}_{V}(\lambda,x,y)=\mathcal{R}(\lambda,x,y)+\sum_{j=1}^{\infty
}\sigma_{j}\mathcal{R}(\lambda,x,a_{j})\mathcal{R}_{V}(\lambda,a_{j},y).
\label{Krein-infty1}%
\end{equation}
Equation (\ref{Krein-infty1}) shows that to estimate the function
$y\rightarrow$ $\left\Vert \mathcal{R}_{V}(\lambda,\cdot,y)\right\Vert _{2}$
it is enough to estimate the quantity $|\mathcal{R}_{V}(\lambda,a_{j},y)|$ for
$j=1,2,...$ etc. Indeed, since $\left\Vert \mathcal{R}(\lambda,\cdot
,y)\right\Vert _{2}$ does not depend on $y$, we get
\begin{equation}
\left\Vert \mathcal{R}_{V}(\lambda,\cdot,y)\right\Vert _{2}\leq\left\Vert
\mathcal{R}(\lambda,\cdot,y)\right\Vert _{2}\left(  1+\beta\sum_{j=1}^{\infty
}|\mathcal{R}_{V}(\lambda,a_{j},y)|\right)  . \label{Norm-bound}%
\end{equation}
Choosing $x=a_{i},$ $i=1,2,...,$ in equation (\ref{Krein-infty1}) and setting
$\mathcal{R}(\lambda,a_{i},a_{i})=\mathcal{R}(\lambda)$ we obtain
\begin{equation}
\mathcal{R}_{V}(\lambda,a_{i},y)=\frac{\mathcal{R}(\lambda,a_{i},y)}%
{1-\sigma_{i}\mathcal{R}(\lambda)}+\sum_{j:\text{ }j\neq i}\frac{\sigma
_{j}\mathcal{R}(\lambda,a_{j},a_{i})\mathcal{R}_{V}(\lambda,a_{j},y)}%
{1-\sigma_{i}\mathcal{R}(\lambda)}. \label{Krein-infty eq}%
\end{equation}
\underline{Step III.} Applying in equation (\ref{Krein-infty eq}) the
inequality%
\[
\left\vert \sum_{j=1}^{\infty}Z_{j}\right\vert ^{s}\leq\sum_{j=1}^{\infty
}\left\vert Z_{j}\right\vert ^{s},\text{ }Z_{j}\in\mathbb{C},\text{ }%
0<s\leq1,
\]
we will get%
\begin{align*}
\left\vert \mathcal{R}_{V}(\lambda,a_{i},y)\right\vert ^{s}  &  \leq\left\vert
\frac{\mathcal{R}(\lambda,a_{i},y)}{1-\sigma_{i}\mathcal{R}(\lambda
)}\right\vert ^{s}+\sum_{j:\text{ }j\neq i}\left\vert \frac{\sigma
_{j}\mathcal{R}(\lambda,a_{j},a_{i})}{1-\sigma_{i}\mathcal{R}(\lambda
)}\right\vert ^{s}\left\vert \mathcal{R}_{V}(\lambda,a_{j},y)\right\vert
^{s}\\
&  \leq\left\vert \frac{\mathcal{R}(\lambda,a_{i},y)}{1-\sigma_{i}%
\mathcal{R}(\lambda)}\right\vert ^{s}+\beta^{s}\sum_{j:\text{ }j\neq
i}\left\vert \frac{\mathcal{R}(\lambda,a_{j},a_{i})}{1-\sigma_{i}%
\mathcal{R}(\lambda)}\right\vert ^{s}\left\vert \mathcal{R}_{V}(\lambda
,a_{j},y)\right\vert ^{s}.
\end{align*}
Taking the expectation over $\{\sigma_{i}\}$ we obtain the following
inequality%
\begin{align}
\mathbb{E}\left\vert \mathcal{R}_{V}(\lambda,a_{i},y)\right\vert ^{s}  &
\leq\mathbb{E}\left\vert \frac{1}{1-\sigma_{i}\mathcal{R}(\lambda)}\right\vert
^{s}\left\vert \mathcal{R}(\lambda,a_{j},y)\right\vert ^{s}%
\label{Expectation'}\\
&  +\beta^{s}\sum_{j:\text{ }j\neq i}\mathbb{E}\left\vert \frac{\mathcal{R}%
_{V}(\lambda,a_{j},y)}{1-\sigma_{i}\mathcal{R}(\lambda)}\right\vert
^{s}\left\vert \mathcal{R}(\lambda,a_{j},a_{i})\right\vert ^{s}.\nonumber
\end{align}
\underline{Step IV.} Due to equation (\ref{Krein}) the random variable
$\mathcal{R}_{V}(\lambda,a_{j},y)$ can be represented in the form%
\[
\mathcal{R}_{V}(\lambda,a_{j},y)=\mathcal{R}_{V^{\prime}}(\lambda
,a_{j},y)+\frac{\sigma_{i}\mathcal{R}_{V^{\prime}}(\lambda,a_{j}%
,a_{i})\mathcal{R}_{V^{\prime}}(\lambda,a_{i},y)}{1-\sigma_{i}\mathcal{R}%
_{V^{\prime}}(\lambda,a_{i},a_{i})}:=\frac{a\sigma_{i}+b}{c\sigma_{i}+d}%
\]
where the random variables $V^{\prime}=$ $-\sum_{k:\text{ }k\neq i}\sigma
_{k}(\omega)\delta_{a_{k}}$, $a,b,c\mathcal{\ }$and $d$ do not dependent on
$\sigma_{i}$ (but they of course depend on the truncated sequence
$\{\sigma_{k}:k\neq i\}$). This observation and the following two general
inequalities from Molchanov's lectures \cite[Chapter II, Lemma 2.2]%
{Molchanov1}): There exist constants $c_{0},c_{1}>0$ such that for all complex
numbers $a,b,c,d,\sigma^{\prime}$
\[
\int_{0}^{1}\frac{d\sigma}{\left\vert \sigma-\sigma^{\prime}\right\vert ^{s}%
}\leq\frac{c_{0}}{1-s},\text{ for all }0<s<1,
\]
and%
\[
\int_{0}^{1}\left\vert \frac{a\sigma+b}{c\sigma+d}\right\vert ^{s}%
\frac{d\sigma}{\left\vert \sigma-\sigma^{\prime}\right\vert ^{s}}\leq
c_{1}\int_{0}^{1}\left\vert \frac{a\sigma+b}{c\sigma+d}\right\vert ^{s}%
d\sigma,\text{ for all }0<s<1/2,
\]
yield the following lemma, which is the fundamental point of our reasons.

\begin{lemma}
\label{Decoupling lemma}(Decoupling lemma) There exist constants $C_{0}%
,C_{0}^{\prime}>0$ which depend on $s,\alpha,\beta$ and $k$ such that the
inequalities%
\[
\mathbb{E}\left\vert \frac{1}{1-\sigma_{i}\mathcal{R}(\lambda)}\right\vert
^{s}\leq C_{0}%
\]
and
\[
\mathbb{E}\left\vert \frac{\mathcal{R}_{V}(\lambda,a_{j},y)}{1-\sigma
_{i}\mathcal{R}(\lambda)}\right\vert ^{s}\leq C_{0}^{\prime}\mathbb{E}%
\left\vert \mathcal{R}_{V}(\lambda,a_{j},y)\right\vert ^{s}%
\]
hold for all $0<s<1/2$ and all $\lambda\in%
\mathbb{C}
_{+}$ such that $\operatorname{Re}\lambda\in\mathcal{I}_{k}$.
\end{lemma}

\underline{Step V.} For any fixed $y\in X$ and $\lambda$ as above let us
denote $\psi_{i}:=$ $\mathbb{E}\left\vert \mathcal{R}_{V}(\lambda
,a_{i},y)\right\vert ^{s}$. Applying Decoupling lemma to inequality
(\ref{Expectation'}) and setting $C_{1}:=\beta^{s}C_{0}^{\prime}$ we get an
infinite system of inequalities%
\[
\psi_{i}\leq C_{0}\left\vert \mathcal{R}(\lambda,a_{i},y)\right\vert
^{s}+C_{1}\sum_{j:\text{ }j\neq i}\left\vert \mathcal{R}(\lambda,a_{j}%
,a_{i})\right\vert ^{s}\psi_{j}.
\]
In the vector form this system reads as follows%
\[
\psi_{i}\leq g_{i}+\left(  \mathcal{A}\psi\right)  _{i}\text{, }i=1,2,...,
\]
where $\psi=(\psi_{i})$, $g=(g_{i})$ has entries $g_{i}=C_{0}\left\vert
\mathcal{R}(\lambda,a_{i},y)\right\vert ^{s},$ and where $\mathcal{A}$ is an
infinite matrix with non-negative entries $\mathfrak{a}_{ij}=C_{1}\left\vert
\mathcal{R}(\lambda,a_{j},a_{i})\right\vert ^{s}$ if $i\neq j$ and $0$ otherwise.

Iterating formally this infinite system of inequalities we get%
\[
\psi_{i}\leq g_{i}+\left(  \mathcal{A}g\right)  _{i}+\left(  \mathcal{A}%
^{2}g\right)  _{i}+\left(  \mathcal{A}^{3}g\right)  _{i}+...\leq\left(
\left(  I-\mathcal{A}\right)  ^{-1}g\right)  _{i}.
\]
In particular, this would yield the following inequality (one of the
fundamental points in the proof of (\ref{Simon-Wolf-Theorem})),
\begin{equation}
\left\Vert \psi\right\Vert \leq2\left\Vert g\right\Vert \label{psi-gi}%
\end{equation}
given $\mathcal{A}:\mathcal{L\rightarrow L}$ is a bounded linear operator
acting in some Banach space $\mathcal{L}$ of sequences such that%
\begin{equation}
\left\Vert \mathcal{A}\right\Vert \leq1/2. \label{A-norm}%
\end{equation}
For instance, choosing $\mathcal{L=}\{\psi:\left\Vert \psi\right\Vert
=\sum_{i}\mu_{i}\left\vert \psi_{i}\right\vert <\infty\}$ we obtain
\begin{equation}
\sum_{i}\mu_{i}\mathbb{E}\left\vert \mathcal{R}_{V}(\lambda,a_{i}%
,y)\right\vert ^{s}\leq2C_{0}\sum_{i}\mu_{i}\left\vert \mathcal{R}%
(\lambda,a_{i},y)\right\vert ^{s} \label{E-mu-i}%
\end{equation}
given%
\begin{equation}
\left\Vert \mathcal{A}\right\Vert =\sup_{\left\Vert \psi\right\Vert =1}%
\sum_{i}\mu_{i}\left\vert \left(  \mathcal{A}\psi\right)  _{i}\right\vert
\leq1/2. \label{F-mu-i}%
\end{equation}
For $\psi$ such that $\left\Vert \psi\right\Vert =1$ we have%
\begin{align*}
\sum_{i}\mu_{i}\left\vert \left(  \mathcal{A}\psi\right)  _{i}\right\vert  &
\leq\sum_{i}\mu_{i}\sum_{j}\mathfrak{a}_{ij}\left\vert \psi_{j}\right\vert \\
&  =\sum_{j}\mathfrak{\mu}_{j}\left\vert \psi_{j}\right\vert \left(  \sum
_{i}\mu_{i}\mathfrak{a}_{ij}\right)  /\mathfrak{\mu}_{j}\leq\sup_{j}\left(
\sum_{i}\mu_{i}\mathfrak{a}_{ij}\right)  /\mathfrak{\mu}_{j}.
\end{align*}
In particular, inequality (\ref{F-mu-i}) holds whenever
\begin{equation}
\sup_{j}\left(  \sum_{i:\text{ }i\neq j}\mu_{i}\left\vert \mathcal{R}%
(\lambda,a_{j},a_{i})\right\vert ^{s}\right)  /\mu_{j}\leq\frac{1}{2C_{1}}.
\label{R-mu-i}%
\end{equation}
Finally, \ref{R-lambda-ineq} together with (\ref{R-mu-i}) allow us to conclude
that (\ref{A-norm}) holds provided%
\begin{equation}
\sup_{j}\left(  \sum_{i:\text{ }i\neq j}\mu_{i}\frac{1}{\mathrm{d}(a_{j}%
,a_{i})^{s}}\right)  /\mu_{j}\leq\frac{1}{2C_{1}C_{2}}. \label{d-mu-i}%
\end{equation}
\underline{Step VI. }For $\lambda$ as above and $\varepsilon_{j}>0$ which we
will choose later consider events%
\[
A_{j}=\{\left\vert \mathcal{R}_{V}(\lambda,a_{j},y)\right\vert >\varepsilon
_{j}\}.
\]
Applying Chebyshev inequality, we will get, for each $j=1,2,...$, the
following inequality%
\begin{equation}
P(A_{j})\leq\frac{\mathbb{E}\left\vert \mathcal{R}_{V}(\lambda,a_{j}%
,y)\right\vert ^{s}}{\varepsilon_{j}^{s}}. \label{Chebyshev}%
\end{equation}
Equations (\ref{Chebyshev}), (\ref{E-mu-i}) and (\ref{d-mu-i}), yield
\begin{align*}
\sum_{j}P(A_{j})  &  \leq\sum_{j}\frac{\mathbb{E}\left\vert \mathcal{R}%
_{V}(\lambda,a_{j},y)\right\vert ^{s}}{\varepsilon_{j}^{s}}\\
&  \leq2C_{0}\sum_{j}\frac{\left\vert \mathcal{R}(\lambda,a_{j},y)\right\vert
^{s}}{\varepsilon_{j}^{s}}\leq2C_{0}C_{2}\sum_{j}\frac{1}{\mathrm{d}%
(a_{j},y)^{s}\varepsilon_{j}^{s}}%
\end{align*}
provided $\varepsilon_{j}$ are chosen such that
\begin{equation}
\sup_{j}\varepsilon_{j}^{s}\left(  \sum_{i:\text{ }i\neq j}\frac
{1}{\varepsilon_{i}^{s}}\frac{1}{\mathrm{d}(a_{j},a_{i})^{s}}\right)
\leq\frac{1}{2C_{1}C_{2}}. \label{epsilon-0}%
\end{equation}
Let us choose $s=1/2-\delta$ and $\varepsilon_{j}=1/\mathrm{d}(a_{j},y)^{r}$.
Then, truncating if necessary the potential $V$, i.e. passing to the potential
$V^{\prime\prime}=V-V^{\prime}$ with $V^{\prime}$ of finite rank as explained
in Step I, we can assume that the sequence $\varepsilon_{j}$ is a strictly
decreasing sequence. By the ultrametric inequality, we have $\mathrm{d}%
(a_{i},a_{j})=\mathrm{d}(a_{i},y)$. Hence
\begin{align*}
\sup_{j}\varepsilon_{j}^{s}\left(  \sum_{i:\text{ }i\neq j}\frac
{1}{\varepsilon_{i}^{s}}\frac{1}{\mathrm{d}(a_{j},a_{i})^{s}}\right)   &
\leq\sup_{j}\left(  \sum_{i:\text{ }i<j}\frac{1}{\mathrm{d}(a_{j},a_{i})^{s}%
}\right)  +\sup_{j}\varepsilon_{j}^{s}\left(  \sum_{i:\text{ }i>j}%
\frac{\mathrm{d}(a_{i},y)^{rs}}{\mathrm{d}(a_{j},a_{i})^{s}}\right) \\
&  \leq\sup_{j}\left(  \sum_{i:\text{ }i<j}\frac{1}{\mathrm{d}(a_{j}%
,a_{i})^{s}}\right)  +\sup_{j}\varepsilon_{j}^{s}\left(  \sum_{i:\text{ }%
i>j}\frac{1}{\mathrm{d}(a_{j},a_{i})^{(1-r)s}}\right) \\
&  \leq\sup_{j}\left(  \sum_{i:\text{ }i<j}\frac{1}{\mathrm{d}(a_{j}%
,a_{i})^{s}}\right)  +\sup_{j}\left(  \sum_{i:\text{ }i\neq j}\frac
{1}{\mathrm{d}(a_{j},a_{i})^{(1-r)s}}\right)  .
\end{align*}
Thus, truncating the potential $V$ and then choosing $0<\delta<1/2-r/(1-r)$ we
obtain inequality (\ref{epsilon-0}). Moreover, thanks to our choice, the
series $\sum_{j}\varepsilon_{j}$ converges. Hence of course converges the
series $\sum_{j}P(A_{j}).$ Applying the Borel-Cantelli lemma we conclude: For
$P-$a.s. $\omega$ there exists $j_{0}(\omega)$ such that
\begin{equation}
\left\vert \mathcal{R}_{V}(\lambda,a_{j},y)\right\vert \leq\varepsilon
_{j},\text{ for all }j\geq j_{0}(\omega), \label{Borel-Cantelli}%
\end{equation}
holds for all $\lambda\in%
\mathbb{C}
_{+}$ such that $\operatorname{Re}\lambda\in\mathcal{I}_{k}$.

\underline{Step VII.} For $\lambda$ as above, the function $\mathcal{R}%
(\lambda,x,y)=(L-\lambda\mathrm{I})^{-1}\delta_{y}(x)$ belongs to $L^{2}(x,m)$
and, by the homogeneity assumption, its norm $\left\Vert \mathcal{R}%
(\lambda,\cdot,y)\right\Vert _{2}$ does not depend on $y$. Having this in mind
we write inequality (\ref{Norm-bound}) (for the truncated potential
$V^{\prime\prime}$)
\begin{align*}
\left\Vert \mathcal{R}_{V^{\prime\prime}}(\lambda,\cdot,y)\right\Vert _{2}  &
\leq\left\Vert \mathcal{R}(\lambda,\cdot,y)\right\Vert _{2}\left(  1+\beta
\sum_{j}\left\vert \mathcal{R}_{V^{\prime\prime}}(\lambda,a_{j},y)\right\vert
\right) \\
&  \leq\left\Vert \mathcal{R}(\lambda,\cdot,y)\right\Vert _{2}\left(
1+\beta\sum_{j\geq j_{0}(\omega)}\varepsilon_{j}+\beta\sum_{j<j_{0}(\omega
)}\left\vert \mathcal{R}_{V^{\prime\prime}}(\lambda,a_{j},y)\right\vert
\right)  \text{ }%
\end{align*}
which clearly holds for all $\lambda$ as above and for $P-$a.s. $\omega$. As
$\operatorname{Im}\lambda\downarrow0$ we get finite limit for $P-$a.s.
$\omega$ and for each $\lambda\in$ $\mathcal{I}_{k}$ which does not belong to
some exceptional set $\mathcal{I}_{k}(\omega)\subset\mathcal{I}_{k}$ of
Lebesgue measure zero (the exceptional set appears because we pass to the
boundary values of the Cauchy-Stieltjes integrals $\mathcal{R}_{V^{\prime
\prime}}(\lambda,a_{j},y)$, $j<j_{0}(\omega)$, as explained in Theorem
\ref{AizMol}). This is precisely what we claim in equation
(\ref{Krein-infty eq}). The proof is finished.
\end{proof}

\bigskip

\noindent
Alexander Bendikov,
Institute of Mathematics, Wroclaw University, Wroclaw, Poland\\
\emph{E-mail address: bendikov@math.uni.wroc.pl}

\bigskip

\noindent
Alexander Grigor'yan,
Department of Mathematics, University of Bielefeld, Bielefeld, Germany, and
Institute of Control Sciences of Russian Academy of Sciences, Moscow, Russian Federation\\
\emph{E-mail address: grigor@math.uni-bielefeld.de}

\bigskip
\noindent
Stanislav Molchanov,
Department of Mathematics, University of North Carolina, Charlotte, NC 28223,
United States of America, and National Research University Higher School of
Economics, Moscow, Russian Federation\\
\emph{E-mail address: smolchan@uncc.edu}

\end{document}